\documentclass[a4paper]{amsart}
\usepackage{amsfonts}
\usepackage[colorlinks=true,citecolor=blue]{hyperref} %
\usepackage{mathptmx}
\usepackage{eucal}
\usepackage{graphicx}
\usepackage{mathrsfs}
\usepackage{amssymb}
\usepackage{amsmath}
\usepackage{amsthm}
\usepackage{enumerate}
\usepackage{verbatim}
\newtheorem{thm}{Theorem}[section]
\newtheorem{cor}[thm]{Corollary}
\newtheorem{lem}[thm]{Lemma}
\newtheorem{prop}[thm]{Proposition}

\newtheorem{ques}[thm]{Question}

\newtheorem*{thma}{Theorem A}
\newtheorem*{thmb}{Theorem B}
\newtheorem*{corc}{Corollary C}
\newtheorem*{thmd}{Theorem D}
\theoremstyle{definition}
\newtheorem{defn}[thm]{Definition}
\newtheorem{exam}[thm]{Example}
\newtheorem{rem}[thm]{Remark}
\numberwithin{equation}{section}

\newcommand{\seq}[1]{\langle #1 \rangle}
\newcommand{\E}{\mathcal{E}}

\newcommand{\N}{\mathbb{N}}
\newcommand{\Od}{\mathcal{O}}

\newcommand{\Z}{\mathbb{Z}}
\newcommand{\Lom}{\mathcal{L}}
\newcommand{\eps}{\varepsilon}
\newcommand{\htop}{h_{\mathrm{top}}}

\newcommand{\Orb}{\mathrm{Orb}}
\newcommand{\dbl}{d_{\mathrm{BL}}}
\DeclareMathOperator{\supp}{supp}
\DeclareMathOperator{\diam}{diam}
\DeclareMathOperator{\dist}{dist}
\DeclareMathOperator{\BL}{BL}
\DeclareMathOperator{\CR}{CR}
\DeclareMathOperator{\Per}{Per}

\begin{document}
\title[Invariant measures and shadowing]{Properties of invariant measures in dynamical systems with the shadowing property}

\author[J. Li]{Jian Li}
\date{\today}
\address[J.~Li]{Department of Mathematics, Shantou University, Shantou, Guangdong 515063, P.R. China}
\email{lijian09@mail.ustc.edu.cn}

\author[P. Oprocha]{Piotr Oprocha}
\address[P. Oprocha]{
	AGH University of Science and Technology,
	Faculty of Applied Mathematics,
	al. Mickiewicza 30,
	30-059 Krak\'ow,
	Poland
	 -- and --
National Supercomputing Centre IT4Innovations, Division of the University of Ostrava,
Institute for Research and Applications of Fuzzy Modeling,
30. dubna 22, 70103 Ostrava,
Czech Republic	}
\email{oprocha@agh.edu.pl}

\begin{abstract}
For dynamical systems with the shadowing property,
we provide a method of approximation of invariant measures by ergodic measures
supported on odometers and their almost 1-1 extensions.
For a topologically transitive system with the shadowing property,
we show that ergodic measures supported on odometers are dense in the space of invariant measures,
and then ergodic measures are generic in the space of invariant measures.
We also show that for every $c\geq 0$ and $\eps>0$ the collection of ergodic measures
(supported on almost 1-1 extensions of odometers) with entropy between $c$ and $c + \eps$
is dense in the space of invariant measures with entropy at least $c$.
Moreover, if in addition the entropy function is upper semi-continuous,
then for every $c\geq 0$ ergodic measures with entropy $c$
are generic in the space of invariant measures with entropy at least $c$.
\end{abstract}
\maketitle
 
\section{Introduction}
The concepts of the specification and shadowing properties were born during studies
on topological and measure-theoretic properties of Axiom A diffeomorphisms
(see \cite{BowSpec} and \cite{BowPOTP} by Rufus Bowen,
who was motivated by some earlier works of Anosov and Sinai).
Later it turned out that there are many very strong connections between specification properties
and the structure of the space of invariant measures.
For example in \cite{Sig1,Sig2}, Sigmund showed that
if a dynamical system has the periodic specification property,
then: the set of measures supported on periodic points is dense in the space of invariant measures;
the set of ergodic measures, the set of non-atomic measures, the set of measures positive on all open sets,
and the set of measures vanishing on all proper closed invariant subsets are complements
of sets of first category in the space of invariant measures;
the set of strongly mixing measures is a set of first category in the space of invariant measures.

A classical result of Bowen (see \cite[Proposition 23.20]{DGS}) states that
if a dynamical system with the shadowing property is expansive and topologically mixing
then it has the periodic specification property.
It is worth mentioning here that expansiveness of dynamics allows to conclude
the uniqueness of some approximate trajectories and the existence of periodic points in the above result.
Therefore it is interesting to expect some connections between
the shadowing property (without expansiveness) and properties of the space of invariant measures.

In general, it may happen that a dynamical system with the shadowing property
does not have any periodic point, however
Moothathu proved in \cite{M11} that the collection of uniformly recurrent points is dense in
the non-wandering set, and later Moothathu and Oprocha proved in \cite{MO13}
that the collection of regularly recurrent points is dense in the non-wandering set.
They also asked in \cite{MO13} whether there always exists a point whose orbit closure is an odometer.
Here we give a positive answer to this question.
In fact, we show that if a dynamical system has the shadowing property,
then the collection of points whose orbit closures are odometers is dense
in the non-wandering set.
If in addition the dynamical system is transitive,
then we can show that collection of ergodic measures which are supported on odometers is dense
in the space of invariant measures.
Our first main result is the following.

\begin{thma}\label{thm:A}
Suppose that a dynamical system $(X,T)$ is transitive and has the shadowing property.
Then the collection of ergodic measures which are supported on odometers
is dense in the space of invariant measures of $(X,T)$.
In particular ergodic measures are generic in the space of invariant measures.
\end{thma}

It is shown in~\cite{KLO} that
if the collection of ergodic measures is dense in the space of invariant measures
then the dynamical system must be transitive.
So transitivity is a necessary assumption in Theorem~A.

Using Theorem A, we can show that
if a dynamical system $(X,T)$ is transitive with $X$ being infinite and has the shadowing property
then the space of invariant measure has many properties in common with the case of
dynamical systems with the specification property.
For example the set of non-atomic measures and the set of measures positive on all open sets
are complements of sets of first category in the space of invariant measures;
the set of strongly mixing measures is a set of first category in the space of invariant measures.
A direct consequence of Theorem A is that the collection of ergodic measures with zero entropy
is dense in the space of invariant measures.
If in addition the entropy function (with respect to invariant measures) is upper semi-continuous,
then ergodic measures with zero entropy are generic in the space of invariant measures.
In \cite{LO13}, we proved that weak mixing together with the
shadowing property imply the specification property with a special kind of regularity in
tracing (a weaker version of periodic specification property).
Using ideas in the proof of Theorem~A, we can show that
the orbit closure of the tracing point is an odometer, answering the Question~1 in~\cite{LO13}.

Another important consequence of the specification property is that,
if the entropy function is upper semi-continuous  then
every invariant measure $\mu$ can be presented as a weak limit of ergodic measures $\mu_n$
and additionally entropies converge,
that is $\lim_{n\to\infty}\mu_n=\mu$ and $\lim_{n\to\infty} h_{\mu_n}(T)=h_\mu(T)$,
e.g. see \cite[Theorem~B]{EKW94}.
It is well known that if a dynamical system is expansive
then the entropy function is upper semi-continuous.
Unfortunately, it many cases it happens that the entropy function is not automatically upper semi-continuous,
hence many authors use some forms of expansiveness to ensure
upper semi-continuity of entropy function and obtain the above convergence, e.g. see \cite{Ruelle}
or weaker versions of expansiveness such as
$h$-expansiveness in \cite{B72} or asymptotically $h$-expansiveness in \cite{M73}.
Note also that if the entropy function is upper semi-continuous
then the system admits an ergodic invariant measure with maximal entropy.

Recently, many authors were interested in weakening specification property
in such a way that the entropy approximation is still possible.
An important example of this type is the paper by Pfister and Sulivan \cite{PS}
who introduced a very weak version of specification property and proved that
it suffices to show results on entropy analogous to these mentioned above
(in particular $\lim_{n\to\infty} h_{\mu_n}(T)\geq h_\mu(T)$).
As we will show, in the case of dynamical systems with the shadowing property
it is always possible to obtain the convergence of entropy, even in the case that
there does not exist any measure of maximal entropy (the entropy function is not upper semi-continuous).
Strictly speaking, the second main result of the paper is the following.
\begin{thmb}
Suppose that a dynamical system $(X,T)$ is transitive and has the shadowing property.
Then for every invariant measure $\mu$ on $(X,T)$ and every $0\leq c\leq h_\mu(T)$ 
there exists a sequence of ergodic measures  $(\mu_n)_{n=1}^\infty$ of $(X,T)$ supported
on almost 1-1 extensions of odometers such that
$\lim_{n\to \infty}\mu_n=\mu$ and $\lim_{n\to \infty}h_{\mu_n}(T)=c$.
\end{thmb}

A direct consequence of Theorem B is the following corollary.

\begin{corc}
Suppose that a dynamical system $(X,T)$ is transitive and has the shadowing property.
The following conditions hold:
\begin{enumerate}
\item for every $0\leq c<\htop(T)$ and $\eps>0$ the collection of ergodic measures,
supported on almost 1-1 extensions of odometers, with entropy between $c$ and $c + \eps$
is dense in the space of invariant measures with entropy at least $c$.
\item if the entropy function is upper semi-continuous,
then for every $0\leq c<\htop(T)$ ergodic measures with entropy $c$
are generic in the space of invariant measures with entropy at least $c$.
\end{enumerate}
\end{corc}

Note that the conclusion of Corollary C is much stronger than the condition (D) in \cite{IP84},
that is the graph of the restriction of the entropy function to ergodic measures
is dense in the graph.
In \cite[Theorem 1]{C15}, Comman gave several equivalent conditions to
the condition (D) and explained their applications in large deviation theory.
So if a dynamical system $(X,T)$ is transitive and has the shadowing property
then it has some consequences in large deviation theory.

The paper is organized as follows.
In Section~\ref{sec:2}, we give some preliminaries.
In Section~\ref{sec:3} we show the existence of odometers in dynamical systems with the shadowing property
and provide a method of approximation of invariant measures by ergodic measures supported on odometers.
Then we prove Theorem A and its consequences.
In Section~\ref{sec:almost11}, we provide a delicate method of approximation of invariant measures
by ergodic measures supported on almost 1-1 extensions of odometers with arbitrary lower entropy.
Theorem B is proved there. Section~\ref{sec:5} contains examples of systems with the shadowing property but without the specification property, highlighting
possible applications of previously mentioned results.
In Section~\ref{sec:6}, we provide the following example.

\begin{thmd}\label{thm:shad:mme}
	There exists a transitive dynamical system $(X,T)$ with the shadowing property such that
	$ h_\mu(T)<\htop(T)$ for every invariant measure $\mu\in M_T(X)$.
\end{thmd}

\section{Preliminaries}\label{sec:2}
In this section we provide main definitions and prove some auxiliary results which will be used later.

Throughout this paper, let $\N$, $\N_0$, $\Z$ and $\mathbb{R}$
denote the set of all positive integers, non-negative integers, integers and real numbers, respectively.
The cardinality of a set $A$ is denoted $|A|$.

\subsection{Topological dynamics}
By a \emph{(topological) dynamical system} we mean a pair
$(X,T)$, where $X$ is a compact metric space and $T\colon X\to X$ is a continuous map.
The metric of $X$ is usually denoted by $d$. If $|X|=1$ then we say that the dynamical system $(X,T)$ is \textit{trivial}.

For a point $x\in X$, the \emph{orbit} of $x$, denoted by  $\Orb(x,T)$,
is the set $\{T^nx\colon n\in\N_0\}$,
and the \emph{$\omega$-limit set} of $x$, denoted by $\omega(x,T)$,
is the set of limits points of the sequence $(T^n(x))_{n\in\N_0}$.

A subset $A$ of $X$ is \emph{$T$-invariant} (or simply \emph{invariant}) if $T(A)\subset A$.
If $A$ is a closed $T$-invariant subset of $X$, then $(A,T|_A)$ also is a dynamical system.
We will call it a \emph{subsystem} of $(X,T)$.
If there is no ambiguity, for simplicity we will write $T$ instead of $T|_A$.
A dynamical system $(X,T)$ is called \emph{minimal}
if it does not contain any non-empty proper subsystem.
It is easy to see that a dynamical system $(X,T)$ is minimal
if and only if $\overline{\Orb(x,T)}=X$ for every $x\in X$.
A dynamical system $(X,T)$ is \emph{transitive} if for any two nonempty open sets $U,V\subset X$
there is $n>0$ such that $T^n(U)\cap V\neq \emptyset$;
\emph{(topologically) weakly mixing} if the product system $(X\times X,T\times T)$ is transitive;
	\textit{(topologically) mixing} if for any two nonempty open sets $U,V\subset X$ there is $N>0$
	such that $T^n(U)\cap V\neq \emptyset$ for every $n\geq N$.

A point $x\in X$ is \emph{periodic} with the least period $n$, if
$n$ is the smallest positive integer satisfying $T^n(x) = x$;
\emph{regularly recurrent} if for every open neighborhood $U$ of $x$ there exists $k\in\N$
such that $T^{kn}(x)\in U$ for all $n\in\N_0$;
\emph{uniformly recurrent} if for every open neighborhood $U$ of $x$ there exists $N\in\N$
such that for every $n\in\N_0$ there is $k\in [n,n+N]$ such that $T^k(x)\in U$;
\emph{recurrent} if for every neighborhood $U$ of $x$ there exists $k\in\N$
such that $T^k(x)\in U$;
and \emph{non-wandering}
if for every neighborhood $U$ of $x$ there exist $k\in\N$ and $y\in U$ such that $T^k(y)\in U$.
It is easy to see that a point $x\in X$ is uniformly recurrent
if and only if $(\overline{\Orb(x,T)},T)$ is minimal, and it is
recurrent if and only if $(\overline{\Orb(x,T)},T)$ is transitive.
The set of all non-wandering points of $(X,T)$ is denoted as $\Omega(X,T)$.
Observe that $\Omega(X,T)$ is closed and $T$-invariant.
If $\Omega(X,T)=X$, the dynamical system  $(X,T)$ is said to be \emph{non-wandering}.

Directly from definitions we have the following relation between the above properties
$$
\text{periodic}\Rightarrow \text{regularly recurrent}\Rightarrow \text{uniformly recurrent}\Rightarrow \text{recurrent}\Rightarrow \text{nonwandering}
$$
and it is not hard to provide examples showing that none of these implications can be reversed.

Let $(X,T)$ and $(Y,S)$ be two dynamical systems.
If there is a continuous surjection $\pi\colon X\to Y$ with $\pi\circ T=S\circ \pi$,
we say that $\pi$ is a \emph{factor map},
the system $(Y,g)$ is a \emph{factor} of $(X, T)$ or $(X, T)$ is an \emph{extension} of $(Y, S)$.
We say that a factor map $\pi\colon X\to Y$ is \emph{almost 1-1}
if $\{x\in X\colon \pi^{-1}(\pi(x))=\{x\}\}$ is residual in $X$.

A dynamical system $(X, T)$ is \emph{equicontinuous} if for every $\eps>0$ there is $\delta>0$ with
the property that for every two points $x,y\in X$,
$d(x,y)<\delta$ implies $d(T^n(x),T^n(y))<\eps$ for all $n\in\N_0$.

A dynamical system $(X,T)$ is called an \emph{odometer}
if it is equicontinuous and there exists a regularly recurrent point $x\in X$
such that $\overline{\Orb(x,T)}=X$. Note that with this definition, a periodic orbit is also an odometer.
There are several equivalent definitions of odometers. Next we recall one of them. Let $\mathbf{s}=(s_n)_{n\in \N}$ be a nondecreasing
sequence of positive integers such that $s_n$ divides $s_{n+1}$. For each $n\geq 1$ define $\pi_n\colon \Z_{s_{n+1}}\to \Z_{s_n}$ by the natural formula $\pi_n(m)=m\; (\text{mod }s_n)$ and let $G_\mathbf{s}$ denote the following inverse limit
\[
G_\mathbf{s}=\varprojlim_n(\Z_{s_n}, \pi_n)=
\Bigl\{x\in\prod_{i=1}^\infty \Z_{s_n}: x_{n}=\pi_n(x_{n+1})\Bigr \},
\]
where each $\Z_{s_n}$ is given the discrete topology, and on $\prod_{i=1}^\infty \Z_{s_n}$ we have the Tychonoff product topology, hence $G_\mathbf{s}$ is a compact metrizable space. On $G_\mathbf{s}$ we define a natural map $T_{\mathbf{s}}\colon G_\mathbf{s}\to G_\mathbf{s}$ by $T_{\mathbf{s}}(x)_n=x_n+1 \; (\text{mod }s_n)$.

The following theorem binds together notions of odometer and regularly recurrent point (e.g. see \cite[Theorem~5.1]{D05}).
\begin{thm}\label{thm:odometer}
A minimal dynamical system $(X,T)$ is an almost 1-1 extension of some $(G_{\mathbf{s}},T_{\mathbf{s}})$ if and only if there exists a regularly recurrent point $x\in X$ such that $\overline{\Orb(x,T)}=X$.
\end{thm}
It was first proved in \cite{Paul} (see also \cite[Theorem~4.3]{D05}) that an almost 1-1 extension of a minimal equicontinuous system is its maximal equicontinuous factor.
But any equicontinuous system is
its own maximal equicontinuous factor, hence by Theorem~\ref{thm:odometer} each odometer is in fact conjugate to some $(G_\mathbf{s},T_{\mathbf{s}})$, and clearly each $(G_\mathbf{s},T_{\mathbf{s}})$ satisfies our definition of odometer. Then we can view each ``abstract'' odometer as $(G_\mathbf{s},T_{\mathbf{s}})$.
We refer the reader to the survey \cite{D05} by Downarowicz for more details on odometers and their extensions.

\subsection{Chain-recurrence and shadowing property}
An infinite sequence $(x_n)_{n=0}^\infty$ of points in $X$
is a \emph{$\delta$-pseudo-orbit} for a dynamical system $(X,T)$ if $d(T(x_n),x_{n+1})<\delta$ for each $n\in\N_0$.
We say that a dynamical system $(X,T)$ has the \emph{shadowing property}
if for every $\varepsilon>0$ there is a $\delta>0$
such that any $\delta$-pseudo-orbit $(x_n)_{n=0}^\infty$
can be \textit{$\eps$-traced} by a point $y\in X$,
that is $d(T^n(y),x_n)<\varepsilon$ for all $n\in \N_0$.

A \emph{$\delta$-chain} (of length $n$) between $x$ and $y$ is any sequence $(x_i)_{i=0}^{n}$
such that $d(T(x_i),x_{i+1})<\delta$ for each $i=0,\ldots,n-1$ and $x=x_0,y=x_{n}$.
Given two points $x,y\in X$, if for every $\delta>0$ there is a $\delta$-chain form $x$ to $y$
then we write $x\rightsquigarrow y$ and when $x\rightsquigarrow y$ and $y\rightsquigarrow x$
then we write $x\sim y$.
A point $x\in X$ is \emph{chain recurrent} if $x\sim x$ and the set of all such points is
denoted by $\CR(X,T)$. Note that the \emph{chain relation} $\sim$ is an equivalence relation on $\CR(X,T)$.
If for every $x,y\in X$ and every $\delta>0$ there exists a $\delta$-chain form $x$ to $y$,
that is $X=[x]_\sim$ for some $x\in X$,
then $(X,T)$ is called \emph{chain transitive}.
It is easy to see that if a dynamical system is transitive then it is also chain transitive,
and when the dynamical system $(X,T)$ has the shadowing property then the converse implication is true.

Obviously $\Omega(X,T)\subset \CR(X,T)$ and it is also well known that
when $(X,T)$ has the shadowing property then $\Omega(X,T)= \CR(X,T)=\CR(\CR(T),T)$.
We refer to \cite{AH} by Aoki and Hiraide for more detailed exposition on chain-recurrence
and the shadowing property.

We say that  a dynamical system $(X,T)$ satisfies the \emph{periodic specification property}
if for any $\eps > 0$ there exists $M > 0$ such that for any $k \geq 2$,
any $k$ points $x_1, x_1,\dotsc,x_k\in X$,
any non-negative integers $0\leq a_1\leq b_1<a_2\leq b_2<\dotsb<a_k\leq b_k$ with $a_i-b_{i-1}\geq M$
for each $i = 2,3,\dotsc, k$ and any integer $p\geq M+b_k-a_1$,
there exists a periodic point $z \in X$  with $T^p(z) = z$ and
$d( T^j (z), T^j (x_i )) <\eps$ for all $a_i\leq j\leq b_i$ and $1\leq i\leq k$.
We say that $(X,T)$ satisfies the \emph{specification property}
if the point $z$ from the definition of periodic specification property is not
requested to be periodic (hence no condition on $p$).

\subsection{Topological entropy}
For $\eps>0$ and $n\in\N$, a subset $S\subset X$ is called \emph{$(n,\eps)$-separated}
if for any $x\neq y\in S$ there exists $0\leq i<n$ with $d(T^i(x),T^i(y))>\eps$,
and \emph{$(n,\eps)$-spanning} if for any $x\in X$, there exists $y\in S$
such that $d(T^i(x),T^i(y))<\eps$ for $i=0,1,\dotsc,n-1$.
Define
\[s_n(\eps) = \sup\{|S|\colon S \text{ is } (n,\eps)\text{-separated}\},\]
and
\[r_n(\eps) = \inf \{|S|\colon S \text{ is } (n,\eps)\text{-spanning}\}.\]

The \emph{topological entropy} of $(X,T)$, denoted by $\htop(X,T)$, is
\[\htop(X,T) =\lim_{\eps\to 0}\limsup_{n\to\infty} \frac{1}{n}\log s_n(\eps) =
\lim_{\eps\to 0}\limsup_{n\to\infty} \frac{1}{n}\log r_n(\eps).\]
We refer to the monograph \cite{Walters}  by Walters for more information on topological entropy.

\subsection{Symbolic dynamics}
For a positive integer $k$,
the symbolic space $\Sigma_k^+$ over the alphabet $\{0,1,\dotsc,k-1\}$ is the collection of
all infinite sequence of symbols in $\{0,1,\dotsc,k-1\}$ indexed by the non-negative integers $\N_0$.
We write elements of $\Sigma_k^+$ as $x=x_0x_1x_2\dotsb$.
If the alphabet $\{0,1,\dotsc,k-1\}$ is endowed with the discrete topology and
$\Sigma_k^+=\{0,1,\dotsc,k-1\}^{\N_0}$  with the product topology,
then the symbolic space $\Sigma_k^+$ is homeomorphic to the standard Cantor ternary set.
The \emph{shift map} $\sigma$ on the symbolic space $\Sigma_k^+$ is defined by
\[\sigma:\Sigma_k^+\to\Sigma_k^+, \quad x_0x_1\dotsb\mapsto x_1x_2\dotsb.\]
It is clear that $\sigma$ is continuous and surjective.
The dynamical system  $(\Sigma_k^+,\sigma)$ is called the \emph{(one-sided) full shift}.
Any subsystem of the full shift is called a \emph{subshift}.

A \emph{word} of length $n$ is a finite sequence $w=w_0w_1\dotsc w_{n-1}$ of elements of
the alphabet $\{0,1,\dotsc,k-1\}$.
We say that a word $w=w_0w_1\dotsc w_{n-1}$ appears in $x=x_0x_1x_2\dotsb$
if there exists some $j\in\N_0$ such that $x_{j+i}=w_i$ for $i=0,1,\dotsc,n-1$.
For a subshift $C$, denote by $\mathcal{L}_n(C)$ is the collection
of words of length $n$ which appear in some point in $C$.
The \emph{language} of $C$, denote by $\mathcal{L}(X)$,
is the set $\bigcup_{n=1}^\infty \mathcal{L}_n(C)$.
The formula for topological entropy of a subshift $(C,\sigma)$ can be reduced to
\[
    \htop(C,\sigma)=\lim_{n\to\infty}\frac{1}{n}\log|\mathcal{L}_n(C)|.
\]
If $X\subset \Sigma_k^+$ is a subshift and there is a finite set of words $F$ such that
\[
    X=\{x\in \Sigma_k^+ : x_{[i,j]}\not\in F \text{ for all }0\leq i \leq j\}
\]
then we say that $X$ is a \emph{subshift of finite type}.
It was first proved by Walters \cite{Wal} that a subshift $(X,\sigma)$
has the shadowing property if and only it is a subshift of finite type.

Two-sided full shift $\Sigma_k=\{0,\ldots,k-1\}^\Z$ and all related objects are defined analogously.
We use monographs \cite{Kit,Kurka} as a standard reference for symbolic dynamics.

\subsection{The space of measures}

Let $X$ be a compact metric space and $\mathcal{B}$ be the $\sigma$-algebra of Borel subsets of $X$.
Let $M(X)$ denote the set of
all Borel probability measures on the measurable space $(X,\mathcal{B})$.
For a point $x\in X$, we denote the point mass at $x$ by $\delta_x$.
The \emph{support} of a measure $\mu \in M(X)$, denoted by $\supp(\mu)$,
is the smallest closed subset $C$ of $X$ such that $\mu(C)=1$.

Let $C(X,\mathbb{R})$ be the collection of all continuous real-valued functions on $X$ with
the supremum norm $\Vert \,\cdot\, \Vert_\infty$.
By the Riesz representation theorem, there is a one-to-one correspondence
between $M(X)$ and the collection of all
normalized continuous positive linear functionals on $C(X,\mathbb{R})$.
So we can imbed $M(X)$ into the dual space $C(X,\mathbb{R})^*$ of $C(X,\mathbb{R})$.
Under the weak$^*$-topology of $C(X,\mathbb{R})^*$, $M(X)$ is a compact metric space.
For a sequence $(\mu_n)_{n\in\N}$ and $\mu\in M(X)$,
we have that $\mu_n\to \mu$ in the weak$^*$-topology if and only if
$\int fd\mu_n\to \int fd\mu$ for every $f\in C(X,\mathbb{R})$.

Let $\BL(X,d)$ denote the set of all bounded real-valued function $f$ on $X$ which
are Lipschitz, i.e.,
\[\Vert f\Vert_L:=\sup_{x\neq y} \frac{|f(x)-f(y)|}{d(x,y)}<\infty.\]
Since $(X,d)$ is a compact metric space, $\BL(X,d)$ is dense in $C(X,\mathbb{R})$ (e.g. see \cite{D04}).
Let $\Vert f\Vert_{\BL}=\Vert f\Vert_\infty+\Vert f\Vert_L$ and
fix a countable dense sequence $(f_n)_{n\in\N}$ in $\{f\in \BL(X,d)\colon \Vert f\Vert_{BL}\leq 1\}$.
For for $\mu,\nu\in M(X)$, define
\begin{equation}\label{dbl:def}
\dbl(\mu,\nu)=\sum_{n=1}^\infty \frac{1}{2^n}\Bigr\vert\int f_nd\mu-\int f_n d\nu\Bigr\vert.
\end{equation}
Then $\dbl$ is a metric on $M(X)$ and
the topology induced by $d_{\BL}$ coincides with the weak$^*$-topology.

The following lemma is easy to be verified.
\begin{lem}\label{lem:measure-approx}
Let $(X,d)$ be a compact metric space and $\eps>0$.
\begin{enumerate}
  \item \label{enum:measure-approx-1}
  For a sequences $(x_i)_{i=0}^\infty$ of points in $X$ and two finite subsets $A,B$ of $\N_0$, then
\[
  \dbl\Bigl( \frac{1}{|A|}\sum_{i\in A} \delta_{x_i}, \frac{1}{|B|}\sum_{i\in B}\delta_{x_i}\Bigr)
  \leq \frac{|A|+|B|}{|A|\cdot|B|}|A\Delta B| + \frac{\bigl||A|-|B|\bigr|}{|A|\cdot|B|}|A\cap B|.
\]
  \item \label{enum:measure-approx-2}
  For two sequence $(x_i)_{i=0}^{m-1}$ and $(y_i)_{i=0}^{m-1}$ of points in $X$
    if $d(x_i,y_i)<\eps$ for $i=0,1,\dotsc,m-1$, then
\[
 \dbl\Bigl( \frac{1}{m}\sum_{i=0}^{m-1} \delta_{x_i},
        \frac{1}{m}\sum_{i=0}^{m-1} \delta_{y_i}\Bigr)<\eps.
\]
 \item\label{enum:measure-approx-3}
 If $\mu_i,\mu\in M(X)$ are such that $\dbl(\mu_i,\mu)<\eps$ for $i=1,\ldots, K$,
 then for any choice of $\alpha_i\in [0,1]$
 with $\sum_{i=1}^K\alpha_i=1$ we have
\[
    \dbl\Bigl(\sum_{i=1}^K \alpha_i \mu_i,\mu\Bigr)<\eps.
\]
\end{enumerate}
\end{lem}

\subsection{Invariant measures and measure-theoretic entropy}
Let $(X,T)$ be a dynamical system.
We say that a measure $\mu\in M(X)$ is \emph{$T$-invariant} if
$\mu(T^{-1}(B))=\mu(B)$ for every $B\in\mathcal{B}$
and denote by $M_T(X)$ the set of all $T$-invariant measures in $M(X)$.
Recall that $M_T(X)$ is a convex compact subset of $M(X)$
and by the celebrated Krylov-Bogolyubov theorem, $M_T(X)$ is always not empty.

For $x\in X$ and $n\in\N_0$, define the $n$-th \emph{empirical measure} of $x$ as
\begin{equation*}
\mathcal{E}_{n}(x):=\frac{1}{n}\sum_{j=0}^{n-1}\delta_{T^{j}(x)}
\end{equation*}
and observe that any limit points of the sequence $(\mathcal{E}_n(x))$ is $T$-invariant.

An invariant measure $\mu\in M_T(X)$ is ergodic
if the only Borel sets $B$ with $T^{-1}(B)=B$ satisfy $\mu(B)=0$ or $\mu(B)=1$.
Let $M^{erg}_T(X)$ denote the set of ergodic measures in $M_T(X)$.
Then $M^{erg}_T(X)$ coincides with the set of extreme points of $M_T(X)$.
By the Choquet representation theorem, for each $\mu \in M_T(X)$
there exists a Borel probability measure $\tau$ on $M^{erg}_T(X)$
such that (see Remark (2) in page 153 of \cite{Walters})
\begin{equation}
\mu=\int_{M^{erg}_T(X)} \nu d\tau(\nu). \label{eq:erg:decomp}
\end{equation}
We call \eqref{eq:erg:decomp} the \emph{ergodic decomposition} of $\mu$.
The entropy of an invariant measure $\mu\in M_T(X)$ is denoted by $h_\mu(T)$.
One of the most useful tools to deal with entropy is  variational principle:
\[\htop(X,T)=\sup_{\mu\in M_T(X)} h_\mu(T)=\sup_{\mu\in M^{erg}_T(X)}h_\mu(T).\]

A dynamical system $(X,T)$ is called \emph{uniquely ergodic} if $M_T(X)$ is a singleton,
and \emph{strictly ergodic} if it is minimal and uniquely ergodic.

\begin{rem}\label{rem:strictly:erg}
If $(G_\mathbf{s},T_\mathbf{s})$ is an odometer, then we may view $T_{\mathbf{s}}$ as a rotation in a compact metrizable group.
Therefore each odometer is strictly ergodic, with the Haar measure as the only invariant measure (e.g. see \cite[Theorem~6.20]{Walters})
\end{rem}

For a given $\mu\in M_T(X)$, we say that a point $x\in X$ is  \emph{generic} for $\mu$
if $\E_n(x)\to \mu$ as $n\to\infty$.
It may happen in practice that an invariant measure  does not have any generic point,
however by the Birkhoff ergodic theorem, if $\mu$ is ergodic
then $\mu$-almost every point is generic.
If a dynamical system $(X,T)$ is uniquely ergodic,
then obviously every point in $X$ is generic for the unique invariant measure.

We will often use the following fact about invariant measures supported on the
orbit closure of a point.
\begin{lem}\label{lem:measure-in-the-orbit}
Let $(X,T)$ be a dynamical system,
let $x\in X$ and $\mu\in M_T^{erg}\bigl(\overline{\Orb(x,T)}\bigr)$.
Then for every $\eps>0$ and $p,N\in\N$,
there exist $n,m\in\N$ with $p|n$, $p|m$ and $m\geq N$ such that
$\dbl(\E_m(T^n(x)),\mu)<\eps$.
\end{lem}

\begin{proof}
Pick a generic point $z\in \overline{\Orb(x,T)}$ for $\mu$.
Then there exists $N_1>N+p$ such that for each $m>N_1$, $\dbl(\E_m(z),\mu)<\frac{\eps}{4}$.
Fix any $m_1>N_1$ sufficiently large to satisfy
\begin{equation}
\frac{5p}{m_1-p}\leq \frac{\eps}{4}. \label{m1p:e4}
\end{equation}
Since $z\in \overline{\Orb(x,T)}$,
there exists $n_1\in\N_0$ such that $d(T^{n_1+i}(x),T^i(z))<\frac{\eps}{4}$
for $i=0,1,\dotsc,m_1-1$.
By Lemma~\ref{lem:measure-approx}\eqref{enum:measure-approx-2},
$\dbl(\E_{m_1}(T^{n_1}(x)),\E_{m_1}(z))<\frac{\eps}{4}$.
Take integers $n_1\leq n < n_1+p$ and $m_1-p< m \leq m_1$ such that $p|n$ and $p|m$.
Then $m_1+n_1-p \leq n+m \leq n_1+m_1+p$ and
by Lemma~\ref{lem:measure-approx}\eqref{enum:measure-approx-1}
we have
\[
    \dbl\bigl(\E_{m_1}(T^{n_1}(x)),\E_{m}(T^{n}(x))\bigr)
    \leq \frac{m_1+m}{m_1\cdot  m} 2p + \frac{p}{m_1\cdot m}m \leq \frac{5p}{m_1-p}
    <\frac{\eps}{4}
\]
and similarly $\dbl(\E_m(z),\E_{m_1}(z))<\frac{\eps}{4}$, which gives
$\dbl(\E_m(T^n(x)),\mu)<\eps$ finishing the proof.
\end{proof}

The following fact seems to be folklore (see the comment in page 451 of~\cite{EKW94}).
We provide a proof for completeness.

\begin{lem}\label{lem:mes:ent:approx:2}
Let $(X,T)$ be a dynamical system and $\mu\in M_T(X)$. Then for every $\eps>0$
there exist ergodic measures $\mu_i$, for $i=1,2,\dotsc,K$ (not necessarily pairwise distinct),
such that
\[
    \dbl\Bigl(\frac{1}{K}\sum_{i=1}^K\mu_i,\mu\Bigr)<\eps,
\]
and
\begin{enumerate}
  \item if $h_\mu(T)<\infty$, then
\[
    \Bigl\vert \frac{1}{K}\sum_{i=1}^K h_{\mu_i}(T)-h_\mu(T)\Bigr\vert<\eps;
\]
  \item if $h_\mu(T)=\infty$, then
\[
    \frac{1}{K}\sum_{i=1}^K h_{\mu_i}(T)>\frac{1}{\eps}.
\]
\end{enumerate}
\end{lem}

\begin{proof}
We first prove the case $h_\mu(T)<\infty$.
Fix any $\eps>0$. For every $\theta\in M_T(X)$, let
\[V(\theta)=\Bigl\{\nu\in M_T(X)\colon \dbl(\theta,\nu)<\frac{\eps}{8}\Bigr\}.\]
Then $\{V(\theta)\colon \theta\in M_T(X)\}$ is an open over of $M_T(X)$.

By the ergodic decomposition theorem,
there exists a Borel probability measure $\tau$ on $M^{erg}_T(X)$ such that
\begin{equation}
    \mu=\int_{M^{erg}_T(X)} \nu d\tau(\nu). \label{erg:decomp}
\end{equation}
Now by the ergodic decomposition of the entropy (see~\cite[Theorem 8.4 (ii)]{Walters}),
we have
\begin{equation}
    h_\mu(T)=\int_{M^{erg}_T(X)} h_\nu(T) d\tau(\nu)<\infty.\label{erg:decomp:ent}
\end{equation}
Clearly the entropy function $M_T(X)\to [0,\infty]$, $\nu\mapsto h_\nu(T)$ is Borel,
and hence each set $M_n=\{\nu\in M^{erg}_T(X)\colon h_\nu(T)\geq n\}$ for $n=1,2,\ldots$ is Borel.
Since $h_\mu(T)<\infty$, by \eqref{erg:decomp:ent} we have $\tau(M_n)\to 0$ as $n\to\infty$
and hence by Lebesgue's convergence theorem also
$\int_{M_n} h_\nu(T)d\tau(\nu)\to 0$ as $n\to\infty$.

Fix a positive constant $\xi< \eps/8$ and a sufficiently large positive integer $n$,
so that both $\tau(M_n)<\xi$ and $\int_{M_n} h_\nu(T)d\tau(\nu)<\xi$.
By definition $h_\nu(T)$ is bounded on $M^{erg}_T(X)\setminus M_n$, hence
there exists a Borel partition $\{P_1,P_2,\dotsc,P_k\}$ of $M^{erg}_T(X)\setminus M_n$
such that the oscillation of $h_\nu(T)$ on every $P_i$ is less than $\xi$, that is,
for every $i$,
\[
    \sup_{\nu\in P_i}h_\nu(T) -\inf_{\nu\in P_i}h_\nu(T)<\xi.
\]
Moreover, we require that the partition $\mathcal{P}$ refines the open cover
$\{V(\theta)\colon \theta\in M_T(X)\}$, which is possible since $M_T(X)$ is compact.
Without loss of generality, we assume that $\tau(P_i)>0$ for $i=1,2,\dotsc,k$
(if it is not the case, we modify $M_n$ so that it includes
each set from $\{P_i\colon \tau(P_i)=0\}$).

Put $P_0=M_n$ and denote $\mathcal{P}=\{P_0,P_1,P_2,\dotsc,P_k\}$.
Then $\mathcal{P}$ is a partition of $M^{erg}_T(X)$.
For every $i=1,2,\dotsc,k$,
fix any $\mu_i\in P_i$.
If $\tau(P_0)>0$ then $\inf_{\nu\in P_0} h_\nu(T)<\infty$ and
so we can find $\mu_0\in P_0$ such that $h_{\mu_0}(T)<\inf_{\nu\in P_0} h_\nu(T)+\xi$.
If $\tau(P_0)=0$ we take any $\mu_0\in M_T(X)$.
Finally, define a probability measure
\[
    \eta=\sum_{i=0}^k \tau(P_i)\mu_i.
\]
Since $h_\nu(T)$ is an affine function of $\nu$	(see \cite[Proposition~10.13]{DGS}) we obtain
that
\begin{align*}
|h_\eta(T)-h_\mu(T)|&= \Bigl|\sum_{i=0}^k \tau(P_i) h_{\mu_i}(T)-\int_{M^{erg}_T(X)} h_\nu(T) d\tau(\nu)\Bigr|\\
&\leq \sum_{i=0}^k \int_{P_i} |h_{\mu_i}(T)-h_\nu(T)| d\tau(\nu)\\
&\leq \sum_{i=1}^k \tau(P_i)\xi +\tau(P_0) h_{\mu_0}(T)+\int_{P_0} h_\nu(T) d\tau(\nu)\\
&\leq \xi+2\int_{P_0} h_\nu(T) d\tau(\nu)\leq \xi+2\xi<\frac{\eps}{2}.
\end{align*}

Recall that for every $i=1,2,\dotsc,k$ we have $\mu_i\in P_i$
and $P_i\subset V(\theta_i)$ for some $\theta_i\in M_T(X)$.
So for $i=1,2,\dotsc,k$ we obtain
\[ \dbl\Bigl(\mu_i, \frac{1}{\tau(P_i)}\int_{P_i}\nu d\tau(\nu)\Bigr)<\frac{\eps}{4}.\]
Now by $\tau(P_0)<\frac{\eps}{8}$, we have
\begin{align*}
\dbl(\eta,\mu)&=\dbl\Bigl(\sum_{i=0}^k \tau(P_i)\mu_i,  \sum_{i=0}^k\int_{P_i}\nu d\tau(\nu)\Bigr)\\
&\leq \sum_{i=1}^k \tau(P_i) \dbl\Bigl(\mu_i, \frac{1}{\tau(P_i)}\int_{P_i}\nu d\tau(\nu)\Bigr)
    +2\tau(P_0)  <\frac{\eps}{2}.
\end{align*}
Taking a positive integer $K$ sufficiently large, we can find positive integers
$K_i$ such that $\sum_{i=0}^n K_i=K$,
\[
    \dbl\Bigl(\sum_{i=0}^k\tau(P_i) \mu_i,\sum_{i=0}^k \frac{K_i}{K}\mu_i \Bigr) <\frac{\eps}{2},
\]
and
\[
    \Bigl\vert \sum_{i=0}^k\tau(P_i) h_{\mu_i}(T)-
    \sum_{i=1}^k \frac{K_i}{K} h_{\mu_i}(T)\Bigr\vert<\frac{\eps}{2}.
\]
Put $\nu=\sum_{i=1}^k\frac{K_i}{K}\mu_i$ and observe that

\[
    \dbl(\nu,\mu)\leq \dbl(\mu,\eta)+\dbl(\eta,\mu)<\eps,
\]
and
\[
    |h_{\mu}(T)-h_\nu(T)|\leq |h_{\mu}(T)-h_\eta(T)|+|h_\eta(T)-h_{\nu}|<\eps.
\]
This completes the proof for the case $h_\mu(T)<\infty$.

For the proof in case $h_\mu(T)=\infty$, pick
a finite Borel partition $\mathcal{P}$ of $X$ with $h_\mu(T,\mathcal{P})>\frac{2}{\eps}$.
Vy the ergodic decomposition of the entropy of the partition (see \cite[Theorem 8.4 (i)]{Walters}),
we have
\begin{equation}
h_\mu(T,\mathcal{P})=\int_{M^{erg}_T(X)} h_\nu(T,\mathcal{P}) d\tau(\nu)<\infty.\label{erg:decomp:ent-2}
\end{equation}
Using \eqref{erg:decomp:ent-2} instead of \eqref{erg:decomp:ent},
it is not hard to modify the proof of the case $h_\mu(T)<\infty$ to work in the case $h_\mu(T)=\infty$.
\end{proof}

For a given $F\subset M_T(X)$, we use the notation
\[
    X_{n,F}=\{x\in X\colon \mathcal{E}_n(x)\in F\}.
\]
The following result is a simplified, one-dimensional version of \cite[Proposition~2.1]{PS}
(cf. \cite{EKW94,DOT15}).
It will be very useful in further calculations in Section 4.

\begin{lem}\label{lem:sep}
Let $(X,T)$ be a dynamical system, $\mu \in M_T^{erg}(X)$.
For every $\eps>0$, there exists $\delta>0$ such that for every neighborhood $F\subset M(X)$ of $\mu$ there is $n_F\in \N$
such that for any $n>n_F$,
there exists $\Gamma_n \subset X_{n,F}$ which is $(n, \delta)$-separated and
\begin{enumerate}
  \item if $h_\mu(T)<\infty$, then $|\frac{1}{n}\log|\Gamma_n|-h_\mu(T)|<\eps$;
  \item if $h_\mu(T)=\infty$, then $\frac{1}{n}\log|\Gamma_n|>\frac{1}{\eps}$.
\end{enumerate}
\end{lem}
\begin{rem}\label{rem:sep}
For an invariant measure $\nu$ we can restrict $T$ to the support of $\mu$,
hence in Lemma~\ref{lem:sep} we can additionally assume that $\Gamma_n\subset\supp(\mu)$.
\end{rem}

\section{The space of invariant measures and the shadowing property}\label{sec:3}
The aim of this section is to prove Theorem A and present some of its consequences.
 To make the idea more clear, we outline the proof of Theorem A as follows.
In a dynamical system with the shadowing property,
if a point $x$ that is close to $T^n(x)$,
then by \cite[Theorem 3.2]{MO13} there exists a regularly recurrent point
$z$ that traces the point $x$.
Here we further extend this technique obtaining that $z$ is equicontinuous and hence
the orbit closure of $z$ is an odometer (see Lemma~\ref{lem:odometer}).

Given an ergodic measure $\nu$, we approximate it,
in the weak$^*$-topology, by an empirical measure of a generic point $x$.
Applying Lemma~\ref{lem:odometer} to $x$,
we obtain a regularly recurrent point $z$ whose orbit closure is an odometer
and $z$ is sufficiently close to $x$, so  that
the empirical measure of $z$ is close in the weak$^*$-topology
to the empirical measure of $x$ and thus the ergodic measure $\nu$.

Now, given a general invariant measure $\mu$ whose support is contained in a single chain-recurrent class,
by the ergodic decomposition theorem we first approximate
it by a finite convex combination of ergodic measures and
then approximate each ergodic measure by an empirical measure of a generic point.
As the support of $\mu$ is contained in a single chain-recurrent class,
by the shadowing property we approximate a sequence created by segments of orbits of those generic points by a point $y$
such that combination of their empirical measures are also to empirical measures obtained on the orbit of $y$.
Similar to the case of ergodic measure,
applying Lemma~\ref{lem:odometer} to $y$,
we get an invariant measure supported on an odometer which is
close the measure $\mu$ in the weak$^*$-topology
 (see Theorem~\ref{thm:approx:erg}).

In our construction we will need the following auxiliary lemmas. Our goal is to construct a regularly recurrent point. It will be achieved by a
	properly performed approximation. The main idea is to use induction, replacing in each step point with $d(x,T^n(x))<\delta$
	by a point $x'$ with sufficiently small value of $d(x',T^{n'}(x))$ for some $n'>0$ and with an $n$-periodic behavior up to accuracy $\eps$, that is $d(T^{j}(x'),T^i(x'))\leq 2\eps$ provided that $j-i\in n\N$ and $0\leq i<n$. Note that such an induction can be initiated using any recurrent point. 
\begin{lem}\label{lem21}
Suppose that a dynamical system $(X,T)$ has the shadowing property.
Let $\eps>0$ and $\delta>0$ be provided for $\eps$ by the shadowing property, and let
 $x\in X$. If $d(T^n(x),x)<\delta$ for some $n\in \N$
then for any $\xi>0$ there exist $x'\in X$ and $n'\in \N$ which is divisible by $n$ such that
\begin{enumerate}
\item[(a)] $d(T^{n'}(x'),x')<\xi$, and
\item[(b)] $d(T^{jn+i}(x'),T^i(x))\leq \eps$ for $i=0,1,\dotsc,n-1$ and every $j\geq 0$.
\end{enumerate}
\end{lem}
\begin{proof}
Let $\alpha$ be the periodic $\delta$-pseudo orbit,
\[x,T(x), \dotsc,T^{n-1}(x), x,T(x), \dotsc,T^{n-1}(x),x,T(x), \dotsc,T^{n-1}(x),\dotsc\]
Let $Y$ be the collection of points $y\in X$ such that
$d(T^{jn+i}(y),T^i(x))\leq \eps$ for $i=0,1,\dotsc,n-1$.
Every point $\eps$-tracing the pseudo-orbit $\alpha$ is in $Y$,
hence $Y$ is non-empty, closed and $T^n(Y)\subset Y$.

Take any minimal subset $D$ for $(Y,T^n)$ and fix any $x'\in D\subset Y$. By definition $x'$ is a uniformly recurrent point of $T^n$, therefore for any $\xi>0$ there exists $k\in \N$ such that $d(T^{nk}(x'),x')<\xi$.
If we put $n'=nk$ then $x'$ and $n'$ are as required.
\end{proof}

\begin{lem}\label{lem:odometer}
Suppose that a dynamical system $(X,T)$ has the shadowing property.
Let $\eps>0$ and $\delta>0$ be provided for $\eps/4$ by the shadowing property and let
 $x\in X$. If $d(T^n(x),x)<\delta$ for some $n\in \N$
then there exists a point $z\in X$ such that:
\begin{enumerate}
\item $z$ is regularly recurrent and $(\overline{\Orb(z,T)},T)$ is an odometer,
\item for every $j\geq 0$, $d(T^{jn+i}(z),T^i(x))\leq \eps$ for $i=0,1,\dotsc,n-1$.
\end{enumerate}
\end{lem}
\begin{proof}
We will start the proof with a recursive construction of a sequence of points $(z_k)$.
First, define sequences $(\eps_k),(\xi_k)$ of positive numbers inductively by
\begin{itemize}
\item $\eps_1=\eps/4$ and $\eps_{k+1}=\eps_k/4$,
\item $\xi_{k}$ is provided for $\eps_{k}$ by shadowing.
\end{itemize}
Let $z_1=x$ and $n_1=n$.
By Lemma~\ref{lem21} there exist $z_2\in X$ and $n_2\in\N$ such that
 \begin{enumerate}[(i)]
\item\label{L3:2a} $d(T^{n_2}(z_2),z_2)<\xi_2$,
\item\label{L3:2b} $d(T^{jn_1+i}(z_2),T^i(z_1))\leq \eps_1$ for $i=0,1,\dotsc,n_1-1$ and every $j\geq 0$.
\end{enumerate}
Proceeding inductively, for each $k>1$ we construct a point $z_k\in X$ and an integer $n_k\in\N $
such that
\begin{enumerate}[(i)]\setcounter{enumi}{2}
  \item\label{L3:ka} $d(T^{n_k}(z_k),z_k)<\xi_k$,
  \item\label{L3:kb} $d(T^{jn_{k-1}+i}(z_k),T^i(z_{k-1}))\leq \eps_{k-1}$ for $i=0,1,\dotsc,n_{k-1}-1$ and all $j\geq 0$.
\end{enumerate}
In particular, $d(z_k,z_{k-1})<\eps_{k-1}$ and therefore for $k<m$ we have
\[d(z_k,z_m)\leq \sum_{i=k}^{m-1} d(z_i,z_{i+1})<\sum_{i=k}^{m-1}\eps_{i}<2\eps_{k}.\]
This shows that $(z_k)$ is a Cauchy sequence and
so $z=\lim_{k\to\infty}z_k$ is well defined and for every $k$ we have
$d(z_k,z)\leq \eps_k$.

Fix any $j\geq 0$, $k\geq 1$ and $m>k$.
Since $n_k$ divides $n_{m-1}$, there exists $l\geq 0$ and $j_{m-1}\geq 0$
such that $jn_k=ln_{m-1}+j_{m-1}n_k$.
In fact, we can put $l=\lfloor\frac{jn_k}{n_{m-1}}\rfloor$ and $j_{m-1}=j-\frac{ln_{m-1}}{n_k}$.
By \eqref{L3:kb}, we have
\[
    d(T^{ln_{m-1}+i}(z_m),T^{i}(z_{m-1}))\leq\eps_{m-1}\text{ for }i=0,1,\dotsc,n_{m-1}-1.
\]
In particular,
\[
    d(T^{jn_k+i}(z_m),T^{j_{m-1} n_k+i}(z_{m-1}))\leq\eps_{m-1}\text{ for }i=0,1,\dotsc,n_k-1.
\]
By the same argument,
there is $j_{m-2}\geq 0$  such that
\[
    d(T^{j_{m-1} n_k+i}(z_{m-1}),T^{j_{m-2} n_k+i}(z_{m-2}))\leq \eps_{m-2}
    \text{ for }t=0,1,\dotsc,n_k-1,
\]
and then
\[
    d(T^{j n_k+i}(z_m),T^{j_{m-2} n_k+i}(z_{m-2}))\leq \eps_{m-1}+\eps_{m-2}
    \text{ for }t=0,1,\dotsc,n_k-1.
\]
Repeating this reduction $(m-k)$-times we obtain
\begin{equation}
d(T^{jn_k+i}(z_m),T^i(z_k))\leq \sum_{l=k}^{m-1}\eps_{l}\leq 2\eps_k
 \text{ for }i=0,1,\dotsc,n_k-1.\label{eq21.i}
\end{equation}
Passing with $m\to\infty$ with a fixed index $0\leq i< n_k$ in \eqref{eq21.i},
we get
\begin{equation}
d(T^{jn_k+i}(z),T^i(z_k))\leq 2\eps_k.   \label{eq22.i}
\end{equation}
In particular, putting $i=0$ we obtain
\begin{equation}
d(T^{jn_k}(z),z)\leq 2\eps_k+d(z,z_k)\leq 4\eps_k.
\end{equation}
As $j \geq 0$ and $ k\geq 1$ are arbitrary, $z$ is regularly recurrent.

Let $Z_{n_k}=\overline{\Orb(z,T^{n_k})}$.
Then by \eqref{eq22.i} for every $i\geq 0$, every $z'\in Z_{n_k}$
and every $\gamma>0$ there is $j\geq 0$ such that
$$
d(T^i(z'),T^i(z))\leq d(T^{jn_k+i}(z),T^i(z'))+d(T^{jn_k+i}(z),T^i(z))\leq 2\eps_k+\gamma.
$$
Passing with $\gamma\to 0$ we obtain $d(T^i(z'),T^i(z''))\leq 4\eps_k$ for every $z',z''\in Z_{n_k}$
and every $i\geq 0$.
In particular $\diam Z_{n_k}\leq 4\eps_k$ and
$Z_{n_k}$ is a neighborhood of $z$ in $\overline{\Orb(z,T)}$
because $z$ is a uniformly recurrent point.
This implies that $z$ is an equicontinuous point,
and therefore $(\overline{\Orb(z,T)},T)$ is an odometer,
since $z$ is regularly recurrent.
\end{proof}

\begin{cor}\label{cor:dense-of-odometers}
Let $(X,T)$ be a dynamical system with the shadowing property.
Then the collection of points whose orbit closure is an odometer is dense
in the non-wandering set $\Omega(X,T)$.
\end{cor}
\begin{proof}
Fix any open set $U$ such that $U\cap \Omega(X,T)\neq \emptyset$.
Fix any $\eps>0$ and $y\in U\cap \Omega(X,T)$
such that $B_{2\eps}(y)\subset U$. Let $\delta<\eps$ be provided by shadowing to $\eps$.
Since $y\in \Omega(X,T)$, there is $x\in B_\eps(y)$ and $n>0$ such that $d(T^n(x),x)<\delta$.
By Lemma~\ref{lem:odometer} there is $z\in B_\eps(x)$ such that $(\overline{\Orb(z,T)},T)$ is an odometer.
But then $z\in \Omega(X,T)$ and $z\in B_\eps(x)\subset B_{2\eps}(y)\subset U$
which completes the proof.
\end{proof}

\begin{thm}\label{thm:approx:erg}
Suppose that $(X,T)$ has the shadowing property and $\mu\in M_T(X)$.
If $\supp(\mu)/_\sim$ is a singleton (i.e. $\supp(\mu)$ is contained in a single chain-recurrent class),
then for every $\eps>0$ there is an ergodic measure $\nu$
supported on an odometer such that $\dbl(\nu,\mu)<\eps$.
\end{thm}
\begin{proof}
Let $\delta\in(0,\frac{\eps}{4})$ be provided for $\frac{\eps}{32}$
and $\xi\in(0,\delta)$ by provided for $\frac{\delta}{2}$ by the shadowing property.
Take any finite open cover $V_1,\ldots, V_k$ of $\supp \mu$ by sets with $\diam V_i<\gamma$,
	where $\gamma<\xi/2$ is such that $d(T(x),T(y))<\xi/2$ provided that $d(x,y)<\gamma$.
For each $1\leq i \leq k$ fix a point $x_i\in V_i$ and observe that since $\supp(\mu)/_\sim$ is a singleton, for any $i,j$ there is a $\xi/2$-pseudo orbit $z_0,\ldots, z_r$
from $x_i$ to $x_j$. But if we fix any $x\in V_i$ and $y\in V_j$ then the sequence
$x,z_1,\ldots, z_{r-1},y$ is a $\xi$-chain form $x$ to $y$. Denote by $M$ the maximum of lengths of these chains over all $1\leq i,j \leq k$.
This, by the shadowing property shows that there exists $M>0$ such that
for any $u,v\in \supp(\mu)$ there are $w\in X$ and $0<m\leq M$ with
$d(u,w)<\xi$ and $d(T^{m}w,v)<\xi$.

By Lemma~\ref{lem:mes:ent:approx:2}
there exist $K>0$ and ergodic measures $\mu_i\in M_T(\supp(\mu))$ for $i=1,2,\dotsc,K$
such that
\begin{equation}
  \dbl\Bigl(\frac{1}{K}\sum_{i=1}^K\mu_i,\mu\Bigr)<\frac{\eps}{8}. \label{eq:thm3.4-1}
\end{equation}
For every $i=1,2,\dotsc,K$, choose a generic point $x_i\in \supp(\mu)$ for $\mu_i$.
Choose a positive integer $N$, sufficiently large to satisfy
\begin{equation}
  \dbl(\E_N(x_i),\mu_i)<\frac{\eps}{8}\text{ for }i=1,2,\dotsc,K.  \label{eq:thm3.4-2}
\end{equation}
which immediately gives
$$
 \dbl(\frac{1}{K}\sum_{i=1}^K\E_N(x_i),\mu)<\frac{\eps}{4}.
$$
By the definition of $M$, for $1\leq i< K$ there exist $z_{i}\in X$ and $0<M_{i}\leq M$
such that $d(T^Nx_i,z_{i})<\xi$ and $d(T^{M_{i}}z_{i},x_{i+1})<\xi$.
There also exist $z_{K}\in X$ and $0<M_{K}\leq M$
such that $d(T^Nx_K,z_{K})<\xi$ and $d(T^{M_{K}}z_{K},x_{1})<\xi$.
Observe that $M$ is independent of the choice of $N$,
in particular we may assume that $N$ is sufficiently large to satisfy
$$
M<N \quad \text{ and }\quad \frac{4M}{N}<\frac{\eps}{8}.
$$

Let
\[n=KN+\sum_{i=1}^{K}M_{i}.\]
Then $KN<n\leq K(N+M)$.
Let $\{y_i\colon 0\leq i\leq n-1\}$ be the sequence of points in $X$ defined as follows
\[x_1,T(x_1),\dotsc,T^{N-1}(x_1),z_{1},T(z_{1}),\dotsc,T^{M_{1}-1}(z_{1}),\]
\[x_2,T(x_2),\dotsc,T^{N-1}(x_2),z_{2},T(z_{2}),\dotsc,T^{M_{2}-1}(z_{2}),\]
\[...\quad ...\quad ...\]
\[x_K,T(x_K),\dotsc,T^{N-1}(x_K),z_{K},T(z_{K}),\dotsc,T^{M_{K}-1}(z_{K}).\]
By Lemma~\ref{lem:measure-approx}\eqref{enum:measure-approx-1}, 
we have
\begin{equation}\label{eq:thm3.4-3}
\begin{split}
\dbl\Bigl(\frac{1}{K}\sum_{i=1}^K\E_N(x_i),\frac{1}{n}\sum_{i=0}^{n-1}\delta_{y_i}\Bigr)
    &\leq \frac{\Bigl(KN+\sum\limits_{i=1}^K(N+M_i)\Bigr)\sum\limits_{i=1}^K M_i+KN
    \sum\limits_{i=1}^K M_i}{K N \sum\limits_{i=1}^K(N+M_i)}\\
&\leq \frac{(2N+M)M+MN}{N^2}\leq \frac{4M}{N}<\frac{\eps}{8}.
\end{split}
\end{equation}

Let $\alpha$ be the periodic sequence
\[\alpha=(y_0,y_1,\dotsc,y_{n-1},y_0,y_1,\dotsc,y_{n-1},y_0,\dotsc).\]
Then $\alpha$ is a $\xi$-pseudo-orbit.
Pick a point $y$ which $\frac{\delta}{2}$-traces $\alpha$, that is
\[
  d(T^{jn+i}(y),y_i)<\frac{\delta}{2},
\]
for every $j\geq 0$ and $i=0,1,\dotsc,n-1$.
Then
\[
    d(T^n(y),y)<d(T^n(y),y_0)+d(y_0,y)<\delta,
\]
and by Lemma~\ref{lem:measure-approx}\eqref{enum:measure-approx-2} we have
\begin{equation}
  \dbl\Bigl(\E_n(y),\frac{1}{n}\sum_{i=0}^{n-1}\delta_{y_i}\Bigr)<\frac{\delta}{2}<\frac{\eps}{8}.
\end{equation}
Applying Lemma~\ref{lem:odometer} to $y$, $n$, $\frac{\eps}{32}$ and $\delta$,
there exists a point $z\in X$
such that the system $(\overline{\Orb(z)},T)$ is an odometer and
\[
    d(T^{jn+i}(z),T^i(y))\leq \frac{\eps}{8}
\]
for every $j\geq 0$ and $i=0,1,\dotsc,n-1$.
By Lemma~\ref{lem:measure-approx}\eqref{enum:measure-approx-2}, again we have
\begin{equation}
    \dbl(\E_{jn}(z), \E_{n}(y))\leq\frac{\eps}{8},
\end{equation}
for every $j\geq 1$.
Since each odometer is strictly ergodic, it has a unique invariant measure, say $\nu$, and therefore
we must have $\nu=\lim_{n\to \infty} \E_n(z)$. There exists $j\in\N$ such that
\begin{equation}
\dbl(\E_{jn}(z),v)<\frac{\eps}{8}.
\end{equation}
Combining all previous partial inequalities we obtain that
\[\dbl(\mu,\nu)<\eps\]
which completes the proof.
\end{proof}

Now we are ready to prove Theorem A.
\begin{proof}[Proof of Theorem A]
If $(X,T)$ is transitive, we have $X=[x]_\sim$ for every $x\in X$, and so
the first part of the conclusion is an immediate consequence of Theorem~\ref{thm:approx:erg}.
It is well known that the collection of ergodic measures are always a $G_\delta$ subset $M_T(X)$
(see \cite[Proposition~5.7]{DGS}).
Then the collection of ergodic measures is residual in $M_T(X)$, as it is dense.
\end{proof}

Now we show some consequence of Theorem A, which is an extension of classical results
by Sigmund proved first for dynamical systems
with  the  periodic  specification  property \cite{Sig1,Sig2} (see also \cite{DGS}).

\begin{prop}\label{prop:full-support}
If a dynamical system $(X,T)$ is transitive and has the shadowing property,
the collection of ergodic measures with full support is residual in $M_T(X)$.
\end{prop}
\begin{proof}
By \cite[Proposition 21.11]{DGS} the collection of invariant measures with full support
is either empty or a dense $G_\delta$ subset of $M_T(X)$.
By Corollary~\ref{cor:dense-of-odometers}, it has a dense set of regularly recurrent points
and thus there is an invariant measure with full support, which completes the proof.
\end{proof}

Any point $x\in X$ with periodic $p$ corresponds a unique measure $\mu$ which has mass $\frac{1}{p}$
at each points $x, T(x),\dotsc,T^{p-1}(x)$.
We denote the set of those measures by $\mathcal{P}(p)$.
Sigmund proved that if $(X,T)$ satisfies the periodic specification
property then for each $\ell>0$, $\bigcup_{p\geq\ell}\mathcal{P}(p)$ is dense in $M_T(X)$
(see \cite[Proposition 21.8]{DGS}).
As a dynamical system with the shadowing property may not contain any periodic points,
we should replace measures supported on periodic points by measures
supported on odometers with disjoint components.

We say that a minimal subsystem $(Z,T)$ of $(X,T)$
has a \emph{periodic decomposition of period $p$}
if there exists a point $z\in Z$ such that if we denote $Z_0=\overline{\Orb(z,T^p)}$ then sets
$Z_0,TZ_0,\dotsc,T^{p-1}Z_0$ are pairwise disjoint, $T^pZ_0=Z_0$ and $\bigcup_{i=0}^{p-1}T^i Z_0=Z$.

For $k\in\N$, denote by $\Od(k)\subset M_T(X)$ the collection of invariant measures $\mu$ such that:
\begin{itemize}
	\item $(Z,T)$ is an odometer where $Z=\supp(\mu)$,
	\item $(Z,T)$ has a decomposition of period $p\geq k$ defined by a point $z\in Z$,
	\item $\diam T^i(Z_0)\leq 1/p$ for every $i\geq 0$, where $Z_0=\overline{\Orb(z,T^p)}$ is the set from the definition of periodic decomposition.
\end{itemize}
Clearly if $\mu \in \Od(k)$ then $\mu$ is ergodic (see Remark~\ref{rem:strictly:erg}).

\begin{prop} \label{prop:RR-q-dense}
Let $(X,T)$ be a dynamical system with $X$ being infinite.
If $(X,T)$ is transitive and has the  the shadowing property,
then for every $\ell>1$, $\Od(\ell)$ is dense in $M_T(X)$.
\end{prop}
\begin{proof}
Fix $\ell>1$ and a non-empty open subset $U$ of $M_T(X)$.
By Proposition~\ref{prop:full-support} there is a fully supported measure $\mu\in U$.
Since $\supp(\mu)$ is infinite, by Theorem~\ref{thm:approx:erg}
there exists an ergodic measure $\nu\in U$ such that $\supp(\nu)$ is an odometer
and $|\supp(\nu)|\geq \ell$. If $\supp(\nu)$ is a periodic orbit,
then since it has at least $\ell$ points, $\nu\in \Od(\ell)$.
If $\supp(\nu)$ is infinite, then $\nu \in \Od(p)$ for every $p$, in particular for $p=\ell$.
\end{proof}

\begin{prop}
Let $(X,T)$ be a dynamical system with $X$ being infinite.
If $(X,T)$ is transitive and has the shadowing property,
the set of non-atomic ergodic measures is residual in $M_T(X)$.
\end{prop}
\begin{proof}
For $r\in\N$, let $K_r$ denote the set
\[
    \bigl\{\mu\in M_T(X)\colon \exists x\in X\text{ with }\mu(\{x\})\geq \tfrac{1}{r}\bigr\}.
\]
By \cite[Proposition 21.10]{DGS} the set $K_r$ is closed.
By Proposition~\ref{prop:RR-q-dense}, we know that $K_r$ is nowhere  dense as $K_r\cap \Od(r+1)=\emptyset$.
The set of $T$-invariant measures with an atom is a subset of $\bigcup_{r=1}K_r$,
and therefore of first category.
\end{proof}

\begin{prop}
If a dynamical system $(X,T)$ is transitive and has the shadowing property,
then the collection of invariant measures with zero entropy is dense in $M_T(X)$.
If in addition, the entropy function is upper semi-continuous,
then the collection of invariant measure with zero entropy is residual in $M_T(X)$.
\end{prop}
\begin{proof}
Note that any measure supported on an odometer has zero entropy.
So the first part of the conclusion follows from Theorem A.
Now assume that the entropy function is upper semi-continuous.
For every $n>0$, let
\[
	A_n=\Bigl\{\mu\in M_T(X)\colon h_\mu(T)<\frac{1}{n}\Bigr\}.
\]
By upper semi-continuity, we know that $A_n$ is open in $M_T(X)$ and then the set
\[
	\bigcap_{n=1}^\infty A_n=\{\mu\in M_T(X)\colon h_\mu(T)=0\}
\]
is a dense $G_\delta$ subset of $M_T(X)$.
\end{proof}

The following fact is an adaptation of \cite[Proposition~21.13]{DGS} to our context.
While there may be no periodic points in $X$
we still may use density of odometers provided by Proposition~\ref{prop:RR-q-dense}.
\begin{prop}
If a nontrivial dynamical system $(X,T)$ is transitive and has the shadowing property,
then the set of strongly mixing measures is of first category in $M_T(X)$.
\end{prop}
\begin{proof}
Observe that if $\mu$ is fully supported strongly mixing measure, then $(X,T)$ is topologically mixing.
But it is also well known that the maximal equicontinuous factor of any (weakly) mixing dynamical system is trivial (e.g. see \cite[Proposition~2.45]{Kurka}).
Therefore if $(X,T)$ is an odometer then we are done, because $|X|>1$ and hence $M_T(X)$ does not contain any strongly mixing measure.

Now assume that $(X,T)$ is not an odometer (in particular $X$ is infinite).
Take any two disjoint closed sets $F_1,F_2\subset X$ with nonempty interiors.
Then $\mu(F_1)>0$ and $\mu(F_2)>0$ for any fully supported measure.
For each $n\geq 2$, put
\[
	S(n)=\{\mu\in M_T(X)\colon \mu(F_1)\geq 1/n, \mu(F_2)\geq 1/n, \mu\text{ is strongly mixing}\}.
\]
Note that any strongly mixing measure with full support is contained in $\bigcup_n S(n)$.
We are going to show that each $S(n)$ is nowhere dense.
For each integer $m\geq 1$ denote $V_m=B(F_1,1/m)$. Fix any $n$ and any $\mu\in S(n)$.
For $r,m\in\N$, put
\[
    E[m,r,s]=\bigcap_{j=s}^\infty\Bigl\{\mu\in M_T(X)\colon \mu(V_m\cap T^{-j}(V_m))-\mu(F_1)^2 \leq \frac{1}{2r^2},
     \mu(F_1)\geq \frac 1n, \mu(F_2)\geq \frac 1n\Bigr\}.
\]
Since $V_m$ is open and $F_1$ and $F_2$ are closed,
it is easy to check that each $E[m,r,s]$ is closed (e.g. see \cite[Proposition~2.7]{DGS}).
If $\mu$ is strongly mixing then for every open set $U$ we have
\[
    \lim_{k\to\infty}\mu(U\cap T^{-k}(U))=\mu(U)^2.
\]
Note that $\lim_{m\to \infty} \mu(V_m\setminus F_1)=0$ for every fully supported non-atomic measure and
\[
	\mu(V_m\cap T^{-j}(V_m))\leq \mu(F_1\cap T^{-j}(F_1))+2\mu(V_m\setminus F_1)
\]
which immediately implies that
\[
    S(n)\subset \bigcup_{m=1}^\infty \bigcup_{r=n}^\infty\bigcup_{s=1}^\infty E[m,r,s]
\]
and so it is enough to show that $E[m,r,s]$ is a nowhere dense subset of $M_T(X)$
for every $m,s\geq 1$ and $r\geq n$.
To do so fix any $m$ and $r\geq n$. Take any $\ell>\max\{n,m+1\}$ and any $\nu \in \Od(\ell)$.
By definition, there exists $p\geq \ell$ such that $(\supp(\nu),T)$ has a periodic decomposition
of period $p$ such that diameters of elements in the decomposition has diameter bounded by $1/\ell$.
Let $Z_0=\overline{\Orb(z,T^p)}$.
If $T^j(Z_0)\cap F_1\neq\emptyset$, then $\diam (T^j(Z_0))<1/\ell< 1/m$ and so $T^j(Z_0)\subset V_m$.
Let $C=\bigcup\{T^j(Z_0): T^j(Z_0)\cap F_1\neq \emptyset\}$.
Then $T^p(C)=C$ and $F_1\subset C\subset  V_m$.
This implies that $\nu(V_m\cap T^{-sp}(V_m))\geq \nu(F_1)$
and therefore either $\nu(F_1)<\frac{1}{n}$ or $\nu(F_2)<\frac{1}{n}$ or
\[
	\nu(V_m\cap T^{-sp}(V_m))-\nu(F_1)^2\geq \nu(F_1)(1-\nu(F_1))\geq \nu(F_1)\nu(F_2)
        \geq \frac{1}{n^2}>\frac{1}{2r^2}.
\]
In any case $\nu\not\in E[m,r,s]$, and since by Proposition \ref{prop:RR-q-dense}
the set $\Od(\ell)$ is dense in $M_T(X)$, we obtain that $E[m,r,s]$ is nowhere dense.
Indeed, each set $S(n)$ is  is of first category which shows stat the set of fully supported
strongly mixing measures  is of first category.
By Proposition \ref{prop:full-support} the set of measures with full support is residual,
so its complement is of first category, completing the proof.
\end{proof}

In \cite{LO13}, we showed that weak mixing together with the
shadowing property imply the specification property with a special kind of regularity in
tracing (a weaker version of periodic specification property).
It is left open in \cite{LO13} that whether
the orbit closure of the tracing point can be an odometer.
The following result answers this question positively.

\begin{prop}
Let $(X,T)$ be a weakly mixing system with the shadowing property.
For every $\eps >0$ there exists $M>0$ such that
for any $k \geq 2$, any $k$ points $x_1, x_2, \ldots , x_k \in X$,
any non-negative integers $0\leq a_1 \leq b_1 < a_2 \leq b_2 < \ldots < a_k \leq b_k$
with $a_i - b_{i-1} \geq M$ for each $i = 2, 3,\ldots , k$ and  any $p \geq M + b_k-a_1$,
there exists a point $z \in X$ whose orbit closure is an odometer
and additionally $d(T^{j}(z),T^{np+j}(z)) < \eps$ for every $n,j\geq 0$ and
$d(T^{np+j}(z), T^j(x_i )) < \eps$ for all $a_i\leq  j \leq b_i$, $1 \leq i \leq k$ and $n\geq 0$.
\end{prop}

\begin{proof}
Note that it if we can prove theorem under additional assumption that $a_1=0$,
then we can easily deduce case $a_1>0$. Therefore, let us assume that $a_1=0$.	

Fix any $\eps>0$. Let $\delta>0$ be provided to $\eps/16$ by the shadowing property.
Assume additionally that $\delta<\eps/2$.
Let $\xi$ be such that every $\xi$-pseudo orbit can be $\delta/3$-traced.
Since $(X,T)$ is weakly mixing, by \cite{RW},
there exists $M\in\N$ such that for every $x,y\in X$ and every $m\geq M$
there is a $\xi$-chain from $x$ to $y$ of length $m$.

Fix any $k \geq 2$, any $k$ points $x_1, x_2, \ldots , x_k \in X$,
any non-negative integers $0\leq a_1 \leq b_1 < a_2 \leq b_2 < \ldots < a_k \leq b_k$
with $a_i - b_{i-1} \geq M$ for $i = 2, 3, \ldots , k$ and  any $p \geq M + b_k$.
There exist $\xi$-chains $\alpha_1,\ldots, \alpha_k$ such that
$\alpha_i$ has length $a_i-b_{i-1}-1$, $\alpha_k$ has length $p-b_k-1$,
\[
    \alpha = x_1,\ldots, T^{b_1}(x_1), \alpha_1, T^{a_2}(x_2),\ldots T^{a_k}(x_k),
    \ldots, T^{b_k}(x_k),\alpha_k
\]
is a $\xi$-chain of length $p$ and $\alpha \alpha$ is also $\xi$-chain.
Let $y$ be a point which is $\delta/3$-tracing periodic $\xi$-pseudo orbit
$\alpha\alpha\alpha\ldots$. Then
\[
    d(y,T^p(y))\leq d(y,x_1)+d(T^p(y),x_1)< \delta,
\]
hence by Lemma~\ref{lem:odometer} there exists a point $z$ such that
$(\overline{\Orb(z,T)},T)$ is an odometer and
$d(T^{jp+i}(z),T^i(y))\leq \frac{\eps}{4}$ for every $i=0,1,\ldots, p-1$.
This implies that for every $a_i\leq  j \leq b_i$, $1 \leq i \leq k$ and $n\geq 0$ we have
\[
    d(T^{np+j}(z), T^j(x_i ))\leq d(T^{np+j}(z),T^j(y))+d(T^j(y),T^j(x_i))\leq
    \frac{\eps}{4}+\frac{\delta}{3}<\eps
\]
and for every $n,j\geq 0$, if we present $j=sp+i$, $0\leq i <p$ then
\begin{align*}
    d(T^{np+j}(z),T^{j}(z))&=d(T^{p(n+s)+i}(z),T^{ps+i}(z))\\
    &\leq d(T^{p(n+s)+i}(z),T^i(y))+d(T^i(y),T^{ps+i}(z))\leq \frac{\eps}{4}+\frac{\eps}{4}<\eps.
\end{align*}
The proof is completed.
\end{proof}

\section{Approximation of entropy by ergodic measures supported on almost 1-1
extensions of odometers}\label{sec:almost11}

The aim of this section is to prove Theorem~B, Corollary~C and some related results.
Theorem~B states that if a transitive system has the shadowing property then
one can approximate any invariant measure by a sequence of ergodic measures with lower entropy.
As any ergodic measure supported on odometer has zero entropy,
it is necessary to consider ergodic measures supported on almost 1-1 extensions of odometers.
Here we outline the ideas of Theorem~B for ergodic measure.
Next, this construction is extended to all invariant measures, by technique similar to the proof Theorem~A, but now with aid of the ergodic decomposition of the entropy
(see Lemma~\ref{lem:mes:ent:approx:2}).

Given an ergodic measure $\mu$, by Lemma~\ref{lem:sep}
there exists a separated set $\Lambda$
such that the cardinality of $\Lambda$ is related to the entropy of $\mu$ and
 the empirical measure of any point in the separated set $\Lambda$
is sufficiently close to the ergodic measure $\mu$ in the weak$^*$-topology.
By the shadowing property and a rather complicated construction (see Lemmas~\ref{lem:2} and \ref{lem:3}),
there exists a regularly recurrent point $z$ such that
$z$ is close to any point in the separated set $\Lambda$
and the entropy supported on the orbit closure of $z$ is related to the cardinality of $\Lambda$ (its orbit follows all elements of $\Lambda$).
We carefully prove that $z$ is regularly recurrent, which implies that its orbit closure is an almost 1-1 extensions of an odometer.

Consider an ergodic measure $\nu$ supported on the orbit closure of $z$
such that the entropy of $\nu$ is close to the entropy of the the orbit closure of $z$.
Then by the construction the measure $\nu$ is close to the ergodic measure $\mu$
and their entropies are also close.
If we remove some points in the separated set $\Lambda$ when constructing $z$,
then $\nu$ is still close to $\mu$ but this way upper bound on entropy of $\nu$ can be controlled.
This is main idea leading to approximation of given ergodic measure $\mu$ by ergodic measures
supported on almost 1-1 extensions of odometers with controlled value of entropy.

We start with the following lemma,
which allows us to decrease radius of returns of pseudo-orbits,
keeping the growth rate of the number of orbits at the same time. 
	The idea is similar to the proof of Lemma~\ref{lem21}, however now condition $d(x,T^n(x))<\delta$
	is replaced by $\diam (\Lambda\cup T^n(\Lambda))<\delta$ so that using tracing we are able to follow segments of orbits of different
	points from $\Lambda$, using $\Lambda$ as a "switching" region. This way we will produce a sufficiently separated set with a kind of $n$-periodic behavior up to some fixed accuracy. This will allow us to perform an inductive construction in later proofs. 

\begin{lem}\label{lem:2}
Suppose that $(X,T)$ has the shadowing property.
Let $\eps>0$, $\eta\in (0,\eps/4)$, let $\delta>0$ be provided for $\eta/2$ by the shadowing property.
If a set $\Lambda=\{x_0,\ldots, x_{s-1}\}$ is $(n,\eps)$-separated for some $n>0$ and
$\diam (\Lambda\cup T^n(\Lambda))<\delta$ then for every $\xi>0$
there exist a set $\Lambda'=\{y_0,\ldots,y_{s'-1}\}$ and $n'>0$ such that
\begin{enumerate}[(i)]
\item\label{lo:1} $\Lambda'$ is $(n',\eps-2\eta)$-separated,
\item\label{lo:2} $\diam (\bigcup_{i=0}^\infty T^{in'}(\Lambda'))<\xi$,
\item\label{lo:3} $n'=r n$ for some integer $r>0$ and $s'\geq s^{(1-\xi)r}$,
\item\label{lo:4} for every $i=0,\ldots, s'-1$ and every integer $j \geq 0$,
there is $i_j\in\{0,1,\dotsc,s-1\}$ such that
\[
    d(T^{jn+t}(y_i),T^t(x_{i_j})<\eta
\]
for $t=0,\ldots,n-1$ and additionally
\[
    |\{0\leq j<r : i_j=q\}|>0
\]
for each $q=0,\ldots,s-1$.
\end{enumerate}
\end{lem}
\begin{proof}
Denote by $\Gamma$ the set of points $y$ in $X$ such that
for each integer $j\geq 0$ there is an index $i_j\in\{0,1,\dotsc,s-1\}$ such that
\begin{equation}\label{eq:sec:ij}
    d(T^{jn+k}(y), T^k(x_{i_j}))\leq \eta/2 \quad \text{ for }k=0,\ldots, n-1.
\end{equation}
Observe that $\Gamma\neq \emptyset$ because
for any choice of indexes $i_0,i_1,\ldots \in \{0,1,\ldots,s-1\}$ the sequence
\begin{equation}
    x_{i_0}, T(x_{i_0}), \ldots, T^{n-1}(x_{i_0}), x_{i_1}, T(x_{i_1}),
    \ldots, T^{n-1}(x_{i_1}), x_{i_2}, \ldots\label{eq:periodicpo}
\end{equation}
is a $\delta$-pseudo orbit, and hence is $\eta/2$ traced by some point in $\Gamma$.
It follows directly from the definition that $\Gamma$ is closed and $T^n(\Gamma)\subset \Gamma$.
Recall that $\Lambda$ is $(n,\eps)$-separated and $\eta<\eps/4$
hence if $d(T^{jn+k}(y), T^k(x_{i_j}))\leq \eta/2$
for $k=0,1,\ldots, n-1$ then the index $i_j$ is uniquely determined.
Namely, if there were two different indexes $i_j,{i_j}'$ satisfying the above condition then for some $0\leq k <n$ we would have
$$
\eps \leq d(T^k(x_{i_j}),T^k(x_{{i_j}'}))\leq d(T^k(x_{i_j},T^{jn+k}(y)))+d(T^{jn+k}(y), T^k(x_{{i_j}'}))\leq \eta<\eps
$$
which is a impossible. Therefore with each $y\in \Lambda$ we can associate a well defined sequence
$\pi(y)=i_0i_1\ldots\in \Sigma_s ^+$ provided by \eqref{eq:sec:ij} and that way we obtain
a continuous map
$\pi \colon \Gamma \to \Sigma_s^+$ such that $\sigma \circ \pi = \pi \circ T^n$.
By \eqref{eq:periodicpo} the map $\pi$ is onto.
Let $C\subset \Sigma_s^+$ be a minimal subshift with $\htop(C)>(1-\xi/4)\log (s)$,
e.g. a Toeplitz subshift (see \cite[Theorem~4.77]{Kurka}; cf \cite{Williams}).
We may also assume that each symbol $0,1,\ldots, s-1$ appears in an element of $C$ (thus in all of them).
To satisfy this condition it is enough to assume that $\htop(C)>\log (s-1)$.
Then there exists $N\in \N$ such that for every $l\geq N$ the set of words of length $l$
allowed in the subshift $C$ has the following lower bound on the number of elements
\[
    |\Lom_l(C)|>s^{(1-\xi/2)l}.
\]
Since the subshift $C$ is minimal, increasing $N$ if necessary,
we may assume that each word in $\Lom(C)$ consisting of at least $N$ symbols
contains at least one occurrence of each symbol $0,1,\ldots, s-1$.
Let $D$ be a minimal set for $T^n$ contained in $\pi^{-1}(C)$.
Clearly $\pi(D)=C$ and so
\[
    \htop (T^n|_D)\geq \htop(C)> (1-\xi/4)\log (s).
\]
Let $\gamma<\xi/4$ be provided by shadowing to $\eps'<\min\{\eta/3,\xi/4\}$.
Let $U_1,\ldots, U_p$ be an open cover of $D$ by sets with $\diam (U_i)<\gamma$.
Since $D$ is minimal for $T^n$, for each $1\leq i,j\leq p$ there is  $s_{i,j}\in\N$
such that $D\cap U_i\cap T^{-n s_{i,j}}(U_j)\neq\emptyset$.
Denote $M=\max_{i,j} \{s_{i,j}\}$.

Fix an integer $l$ such that
\[
    l>\frac{8\log(p)}{\xi}\quad\text{ and }\quad l>\frac{4M(1+\xi)}{\xi}.
\]
For each $w\in \Lom_{l}(C)$,
pick exactly one element $z_w\in \pi^{-1}([w]_C)\cap D$. We claim that the set
\[
    Z_l=\{z_w: w\in \Lom_l(C)\}\subset D
\]
is an $(n l,\eps-\eta)$-separated set for $T$.
Fix any distinct $u,w\in \Lom_l(C)$ and let $0\leq j<l$ be such that $w_j\neq u_j$. Then by \eqref{eq:sec:ij}
we obtain that
$$
d(T^{jn+k}(z_w), T^k(x_{w_j}))\leq \eta/2 \quad \textrm{ and }\quad d(T^{jn+k}(z_u), T^k(x_{u_j}))\leq \eta/2  \quad \text{ for }k=0,\ldots, n-1.
$$
But since $w_j\neq u_j$ there is also $0\leq k <n$ such that $d(T^k(x_{w_j}),T^k(x_{u_j}))>\eps$,
and therefore for this particular $k$ we obtain that
$$
d(T^{jn+k}(z_w), T^{jn+k}(z_u))> \eps-\eta/2-\eta/2
$$
which proves the claim.

Furthermore, using the pigeon-hole principle
we immediately obtain that there are sets $U_i, U_j$ in the list $\{U_k\}_{k=1}^p$ such that
\begin{equation}
    |Z_l \cap U_i \cap T^{-nl}(U_j)|\geq \frac{|\Lom_l(C)|}{p^2} > \frac{s^{(1-\xi/2)l}}{p^2}
    =s^{(1-\xi/2)l-2\log(p)}\geq s^{(1-\frac{3\xi}{4})l}. \label{eq:*xil}
\end{equation}
Put $r=l+s_{i,j}$ and $n'=rn$.
Pick a point $v\in D$ such that $v\in U_j$ and $T^{s_{i,j}n}(v)\in U_i$.

Let $\mathcal{A}=\{w\in\Lom_l(C)\colon z_w\in U_i, T^{nl}z_w\in U_j\}$.
For each $w\in \mathcal{A}$ the sequence
\begin{equation}
    z_w, T(z_w),\ldots, T^{nl-1}(z_w),v,T(v),\ldots, T^{s_{i,j}n-1}(v)\label{eq:v}
\end{equation}
is a finite $\gamma$-pseudo orbit and $d(T^{s_{i,j}n}(v),z_w)<\gamma$,
which we extend to an infinite (periodic) $\gamma$-pseudo orbit,
by periodic repetition of the sequence \eqref{eq:v}.
Pick a point $y_w\in X$ which $\eps'$-traces it.
Then the set
\[
    \Lambda'=\{y_w : w\in \mathcal{A} \}
\]
is $(n',\eps-2\eta)$ separated and by \eqref{eq:*xil} contains at least
\[
   |\Lambda'| = |\mathcal{A} |\geq  s^{(1-\frac{3\xi}{4})l}\geq s^{(1-\xi)r}
\]
elements, hence conditions \eqref{lo:1} and \eqref{lo:3} are satisfied.
But clearly $\bigcup_{i=0}^\infty T^{in'}(\Lambda')\subset B_{\eps'}(U_i)$ and so
\[
    \diam \Bigl(\bigcup_{i=0}^\infty T^{in'}(\Lambda')\Bigr)\leq\diam (U_i)+2\eps'
        <\gamma+2\eps'\leq \frac{\xi}{4}+\frac{\xi}{2}<\xi
\]
which shows that \eqref{lo:2} holds.

Finally, fix any $y_w\in \Lambda'$, where by definition $w\in \mathcal{A}$.
Fix any $j=0,\ldots, r-1$.
If $0\leq j<l$ then put $i_j=w_j$ and observe that for $t=0,1,\dotsc,n$ we have
\begin{align*}
    d(T^{jn+t}(y_w),T^t(x_{i_j}))&\leq d(T^{jn+t}(y_w),T^{j n+t}(z_w))+d(T^{jn+t}(z_w),T^t(x_{i_j}))\\
    &\leq \eps'+\eta/2<\eta.
\end{align*}
If $l\leq j<r-1$ then, because $v\in \Gamma$,
there exists $i_j\in\{0,1,\dotsc,s-1\}$ such that
\[
    d(T^{(j-l) n+t}(v),T^t(x_{i_j}))\leq\eta/2
\]
for $t=0,1,\dotsc,n$ and so also in this case
\begin{align*}
	d(T^{jn+t}(y_w),T^t(x_{i_j}))&\leq d(T^{jn+t}(y_w),T^{(j-l) n+t}(v))
        +d(T^{(j-l) n+t}(v),T^t(x_{i_j}))\\
	&\leq \eps'+\eta/2<\eta,
\end{align*}
for $t=0,1,\dotsc,n$. The case $j>r$ follows from the fact that $y_w$ traces a pseudo-orbit obtained
by periodic repetitions of the sequence \eqref{eq:v}.
Additionally, by the choice of $l$, any word $w$ in $\mathcal{A}$
contains at least one occurrence of each symbol
$0,1,\ldots, s-1$ which implies \eqref{lo:4}.
The proof is completed.
\end{proof}

\begin{lem}\label{lem:3}
Suppose that $(X,T)$ has the shadowing property.
Let $\eps>0$, $\eta\in (0,\eps/8)$ and  let $\delta>0$ be provided for $\eta/8$ by the shadowing property.
If the set $\Lambda=\{x_0,\ldots, x_{s-1}\}$ is $(n,\eps)$-separated for some $n>0$  and
$\diam(\Lambda\cup T^n(\Lambda))<\delta$
then there exists a regularly recurrent point $z$ such that
\begin{enumerate}[(a)]
\item\label{lem3:c1} for every $j\geq 0$ there is $i_j\in\{0,1,\dotsc,s-1\}$
	such that $d(T^{jn+t}(z),T^t(x_{i_j}))\leq2\eta$ for $t=0,\ldots,n-1$,
\item\label{lem3:c2} $(1-\eta)\log(s)/n \leq \htop (\overline{\Orb(z,T)},T) \leq \log(s)/n$.
\end{enumerate}
\end{lem}

\begin{proof}
Fix a sequence  $(\lambda_k)_{k=1}^\infty$ of positive numbers such that
$\lambda_1=\eta/4$ and $\prod_{k=1}^\infty (1-\lambda_k)>(1-\eta)$.
Define inductively sequences $(\eps_k)_{k=1}^\infty$, $(\eta_k)_{k=1}^\infty$,
$(\delta_k)_{k=1}^\infty$ of positive numbers by
\begin{itemize}
\item $\eps_1=\eps$, $\eta_1=\eta/8$ and $\delta_1=\delta$,
\item $\eps_{k+1}=\eps_k-2\eta_k$,
\item $\eta_{k+1}=\eta_k/8=\eta/8^{k+1}$,
\item $\delta_{k+1}$ is provided for $\eta_{k+1}$ by shadowing,
\item $\xi_k=\min\{\lambda_{k},\delta_{k+1},\eta_{k+1}\}$.
\end{itemize}
	
We will construct a sequence $(z_k)_{k=1}^\infty$ of points in $X$ and sequences
$(s_k)_{k=1}^\infty$, $(n_k)_{k=1}^\infty$ and $(r_k)_{k=1}^\infty$ of positive integers by induction.
Put $n_1=n$, $s_1=s$ and $\Lambda_1=\Lambda$.
Enumerate the elements of the $\Lambda_1$  by $\Lambda=\{y_0^{(1)},\ldots,y_{s_1-1}^{(1)}\}$.
Observe that $\Lambda_1$ is $(n_1,\eps_1)$-separated and $\diam (\Lambda_1\cup T^{n_1}(\Lambda_1))<\xi_1$,
hence applying Lemma~\ref{lem:2} for  $\Lambda_1$
with parameters $n_1$, $\eps_1$, $\xi_1$, $\eta_1$ and $\delta_1$,
we obtain a set $\Lambda_2=\{y^{(2)}_0,\ldots,y^{(2)}_{s_2-1}\}$
and positive integers $r_1$, $n_2$, $s_2$ such that
(we can always remove some elements from $\Lambda_2$ if there were too many):
\begin{enumerate}[(i)]
  \item $\Lambda_2$ is $(n_2,\eps_2)$-separated,
  \item $\diam(\bigcup_{i=0}^\infty T^{i n_2}(\Lambda_2))<\xi_1<\delta_2$,
  \item $n_2=r_1 n_1$ and $s_1^{(1-\xi_1)r_1}< s_2\leq s_1^{r_1}$,
  \item  for every $i=0,\ldots, s_2-1$ and $j\geq 0$
  there is $i_j\in\{0,1,\dotsc,n_1-1\}$ such that
\[
  d(T^{jn_1+t}(y^{(2)}_i),T^t(y_{i_j}^{(1)}))<\eta_1
\]
  for $t=0,\ldots, n_1-1$ and additionally
\[
    |\{0\leq j<r_1 : i_j=q\}|>0
\]
  for each $q=0,\ldots,s_1-1$.
\end{enumerate}
Proceeding inductively, for each $k>1$ we generate by Lemma~\ref{lem:2}
a set $\Lambda_k=\{y_0^{(k)},\ldots,y_{s_k-1}^{(k)}\}$
and positive integers $r_{k-1},n_k,s_k$ such that
\begin{enumerate}[(i)]\setcounter{enumi}{4}
\item $\Lambda_k$ is $(n_k,\eps_k)$-separated,
\item\label{eq:kb} $\diam (\bigcup_{i=0}^\infty T^{i n_k}(\Lambda_k))<\xi_{k-1}\leq \delta_k$,
\item $n_k=r_{k-1} n_{k-1}$ and $s_{k-1}^{(1-\xi_{k-1})r_{k-1}}<s_k \leq s_{k-1}^{r_{k-1}}$,
\item\label{eq:kd}  for every $i=0,\ldots, s_k-1$ and $j\geq 0$
 there is $i_j\in\{0,1,\dotsc,s_{k-1}-1\}$ such that
\[
  d(T^{j n_{k-1}+t}(y^{(k)}_i),T^t(y^{(k-1)}_{i_j}))<\eta_{k-1}
\]
$t=0,\ldots,n_{k-1}-1$ and  additionally
\[
    |\{0\leq t<r_{k-1} : i_t=q\}|>0
\]
 for each $q=0,\ldots,s_{k-1}-1$.
\end{enumerate}
Put $z_k=y^{(k)}_0$ and observe that for each $k\geq 1$
we have $z_k\in \Lambda_{k}$ and $z_{k+1}\in \Lambda_{k+1}$,
thus
\begin{align*}
	d(z_k,z_{k+1})&\leq \diam(\Lambda_{k})+\diam(\Lambda_{k+1})+d(\Lambda_{k},\Lambda_{k+1})\\
	&\leq\eta_{k}+\eta_{k+1}+\eta_{k}<\eta/8^{k-1}.
\end{align*}
This shows that $(z_k)_{k=1}^\infty$ is a Cauchy sequence and so
$z=\lim_{k\to\infty} z_k$ is well defined.

Fix any $j\geq 0$, $k\geq 1$ and $m>k$.
Since $n_k$ divides $n_{m-1}$, there exists $l\geq 0$ and $j_{m-1}\geq 0$
such that $jn_k=ln_{m-1}+j_{m-1}n_k$.
In fact, we can put $l=[\frac{jn_k}{n_{m-1}}]$ and $j_{m-1}=j-\frac{ln_{m-1}}{n_k}$.
Since $z_m\in \Lambda_m$, by ($m$-d), there exists $i_{m-1}\in\{0,1,\dotsc,s_{m-1}-1\}$
such that
\[
	d(T^{ln_{m-1}+t}(z_m),T^{t}(y_{i_{m-1}}^{(m-1)}))<\eta_{m-1}\text{ for }t=0,1,\dotsc,n_{m-1}-1.
\]
In particular,
\[
	d(T^{jn_k+t}(z_m),T^{j_{m-1} n_k+t}(y_{i_{m-1}}^{(m-1)}))<\eta_{m-1}\text{ for }t=0,1,\dotsc,n_k-1.
\]
By the same argument,
there is $j_{m-2}\geq 0$ and $i_{m-2}\in\{0,1,\dotsc,s_{m-2}-1\}$ such that
\[
	d(T^{j_{m-1} n_k+t}(y_{i_{m-1}}^{(m-1)}),T^{j_{m-2} n_k+t}(y_{i_{m-2}}^{(m-2)}))<\eta_{m-2}
    \text{ for }t=0,1,\dotsc,n_k-1,
\]
and then
\[
	d(T^{j n_k+t}(z_m),T^{j_{m-2} n_k+t}(y_{i_{m-2}}^{(m-2)}))<\eta_{m-1}+\eta_{m-2}
	\text{ for }t=0,1,\dotsc,n_k-1.
\]
Repeating this reduction $(m-k)$-times, we obtain some $j_{k}\in\{0,1,\dotsc,s_{k}-1\}$ such that
\begin{equation}
	d(T^{jn_k+t}(z_m),T^t(y_{j_{k}}^{(k)}))<\sum_{i=k}^{m-1}\eta_{i}\leq \sum_{i=k}^{\infty} \eta/8^{i}
	\leq 2 \eta_k, \text{ for }t=0,1,\dotsc,n_k-1.\label{eq:trace}
\end{equation}
	
By (\ref{eq:kb}), we have $\diam(\Lambda_k)\leq \eta_k$. Then by~\eqref{eq:trace} with $t=0$, we have
\begin{equation}
	d(T^{j n_k}(z_m),z_k)\leq \diam(\Lambda_k)+2\eta_k\leq 3 \eta_k\leq \eta/8^{k-1}.
	\label{eq:34}
\end{equation}
Passing with $m\to\infty$ with a fixed $j\geq 0$ and $k\geq 1$ in \eqref{eq:34}, we get
\begin{equation*}
	d(T^{j n_k}(z),z_k)\leq \eta/8^{k-1},
\end{equation*}
and then
\[
	d(T^{j n_k}(z),z)\leq d(T^{j n_k}(z),z) +d(z,z_k)
        \leq \eta/8^{k-1} +\sum_{i=k}^\infty d(z_k,z_{k+1})\leq \eta/8^{k-2}.
\]
As $j\geq 0$ and $k\geq 1$ are arbitrary, $z$ is regularly recurrent.
	
By \eqref{eq:trace} with $k=1$, we obtain that for every $j\geq 0$ and $m>1$,
there exists $i_{j,m}\in \{0,1,\dotsc,s-1\}$ such that
\[
    d(T^{jn+t}(z_m),T^t(x_{i_{j,m}}))\leq 2 \eta_1, \text{ for }t=0,1,\dotsc,n-1.
\]
There are only finitely many choices of indexes $i_{j,m}$,
hence for each $j$, the sequence $(i_{j,m})_{m=1}^\infty$ has a subsequence of
a fixed value $i_j$. Passing to a limit for $m\to\infty$ over this subsequence, we get
\[
    d(T^{jn+t}(z),T^t(x_{i_{j}}))\leq 2 \eta_1, \text{ for }t=0,1,\dotsc,n-1,
\]
which shows  that \eqref{lem3:c1} holds.

By the same argument, for each $k\geq 2$ and $j\geq 0$,
there exists $i_{j,k}\in\{0,1,\dotsc,s_k-1\}$ such that
\begin{equation}
	d(T^{jn_k+t}(z),T^{t}(y^{(k)}_{i_{j,k}}))<\eta_{k-1} \text{ for }t=0,1,\dotsc,n_k-1.\label{eq:approx}
\end{equation}

Fix any $k\geq 2$ and observe that by \eqref{eq:approx} with $j=0$,
there exists $i_{k}\in\{0,1,\dotsc,s_k-1\}$ such that
\[
	d(T^{t}(z),T^{t}(y^{(k)}_{i_{k}}))<\eta_{k-1} \text{ for }t=0,1,\dotsc,n_k-1.  
\]
For this particular $i_k$, by (\ref{eq:kd}), for each $q=0,1,\dotsc,s_{k-1}-1$,
there exists $j_q\in \{0,1,\dotsc,r_{k-1}-1\}$ such that
\[
    d(T^{j_qn_{k-1}+t}(y^{(k)}_{i_{k}}),T^t(y^{(k-1)}_q))<\eta_{k-1} \text{ for }t=0,1,\dotsc,n_{k-1}-1
\]
and then
\[
    d(T^{j_qn_{k-1}+t}(z),T^t(y^{(k-1)}_q))<2\eta_{k-1} \text{ for }t=0,1,\dotsc,n_{k-1}-1.
\]
By the construction the set
$\Lambda_{k-1}=\{y^{(k-1)}_0,y^{(k-1)}_1, \dotsc, y^{(k-1)}_{s_{k-1}-1}\}$
is $(n_{k-1},\eps_{k-1})$-separated.
So the set
\[
	\{T^{j n_{k-1}}(z)\colon i=0,\ldots, r_{k-1}-1\}
\]
contains an $(n_{k-1},\eps_{k-1}-2\eta_{k-1})$-separated set
consisting of at least $s_{k-1}$-elements.
Note that $\eps_{k-1}-2\eta_{k-1}\geq \eps/2$ and
\begin{align*}
  \log (s_{k-1})&\geq {(1-\xi_{k-2})r_{k-2}}\log (s_{k-2})\geq \ldots
      \geq \log(s_1)\prod_{j=1}^{k-2}(1-\xi_{j})r_j\\
  &\geq (1-\eta) \log(s_1)\frac{n_{k-1}}{n}=\frac{n_{k-1}}{n}(1-\eta)\log(s).
\end{align*}
So we obtain a lower bound for topological entropy of $T$ on $\overline{\Orb(z,T)}$, which is
\[
    \htop (\overline{\Orb(z,T)},T)\geq \limsup_{k\to\infty} \frac{1}{n_{k-1}}\log(s_{k-1})
    \geq \frac{(1-\eta)}{n}\log(s).
\]
	
Now fix any $\xi>0$ and pick a large enough positive integer $k$ such that $4\eta_{k-1}<\xi$.
Fix any positive integer $m$.
We first assume that $m$ is of the form $m=q n_k$ for some $q\geq 1$.
For any $x\in \Gamma$, there is $r\geq 0$ such that
\[
    d(T^{r+t}(z),T^t(x))<\eta_{k-1}
\]
for $t=0,\ldots,m-1$.
Write $r$ as $r=pn_k+l$ for some $0\leq l<n_k$. Then by \eqref{eq:approx}
there are indexes $w_0,\ldots, w_{q}$ such that
\[
    d(T^{(p+i)n_k+t}(z),T^t(y_{w_i}^{(k)}))<2\eta_{k-1}
\]
for $t=0,\ldots, n_{k}-1$ and $i=0,\ldots, q$. If $0<t<n_k-l$, we have
\[
    d(T^{t}(x), T^{l+t}(y_{w_0}^{(k)}))< d(T^{t}(x), T^{pn_k+l+t}(z))+
        d(T^{pn_k+l+t}(z),T^{l+t}(y_{w_0}^{(k)}))<3\eta_{k-1}.
\]
If $n_k-l\leq t < qn_k$, then $t\in [i n_k-l,(i+1)n_k-l)$ for some $i=0,1,\dotsc,q$, and
\[
    d(T^{t}(x), T^{l+t}(y_{w_{i}}^{(k)}))< d(T^{t}(x), T^{pn_k+l+t}(z))+
        d(T^{pn_k+l+t}(z),T^{l+t}(y_{w_{i}}^{(k)}))<3\eta_{k-1}.
\]
Note that there are $n_k$ choices of $l$ and $s_k^{q+1}$ choices of the indexes $w_0,\ldots, w_{q}$.
We get an upper bound of the minimal number of elements of $(m,\xi)$-spanning set
for $(\overline{\Orb(z,T)},T)$.
Strictly speaking, if $S_m$ is an $(m,\xi)$-spanning set for $(\Gamma,T)$ with minimal number of elements,
then
\[
	|S_m|\leq n_k s_k^{q+1}\leq n_k s^{(q+1)\prod_{j=1}^{k-1} r_j}
        =n_k s^{(q+1)n_k/n}.
\]
For general positive integer $m$, there exists $q\in\N_0$ such that $q n_k\leq m<(q+1)n_k$
and in this case $|S_{(i+1)n_k}|\leq |S_m|\leq | S_{in_k}|$.
So
\[
    \limsup_{m\to \infty} \frac{1}{m}\log |S_m|= \limsup_{q\to \infty}\frac{1}{q n_k}\log |S_{(q+1)n_k}|
        \leq \limsup_{q\to \infty}\frac{1}{q n_k} \log(n_ks^{(q+2)n_k/n})
        =\frac{1}{n}\log(s).
\]
But $\xi<0$ is arbitrary, which shows that topological entropy on $\overline{\Orb(z,T)}$
is bounded from above by $\frac{1}{n}\log(s)$, completing the proof.
\end{proof}

If a dynamical system is transitive, then the whole space is the unique chain-recurrent class.
So Theorem B follows from the following result.
\begin{thm}\label{thm:chain-rec:ent}
Suppose that $(X,T)$ has the shadowing property and $\mu\in M_T(X)$.
If $\supp(\mu)/_\sim$ is a singleton, then for every $0\leq c\leq h_\mu(T)$
there exists a
sequence of ergodic measures $(\mu_n)_{n=1}^\infty$
supported on almost 1-1 extensions of odometers such that
$\lim_{n\to \infty}\mu_n=\mu$ and $\lim_{n\to \infty}h_{\mu_n}(T)=c$.
\end{thm}
\begin{proof}
We have the following three cases depending on the entropy of $\mu$.

\noindent \textbf{Case 1:} $h_\mu(T)=0$ or $c=0$. By Theorem~\ref{thm:approx:erg}
we may approximate $\mu$ by measures supported on odometers
and every measures on odometers have zero entropy.

\noindent \textbf{Case 2:} $h_\mu(T)\in (0,\infty)$.
Fix $\eps>0$ and $c\in (0,h_\mu(T)]$.
It is enough to show that there exists an ergodic measure $\nu$
supported on an almost 1-1 extensions of odometer such that
$\dbl(\nu, \mu)<\eps$ and $|h_{\nu}(T)-c|<\eps$.

By Lemma~\ref{lem:mes:ent:approx:2},
there exist ergodic measures $\mu_i$ with $\supp(\mu_i)\subset \supp(\mu)$
for $i=1,2,\dotsc,K$ such that
\begin{equation}
  \dbl\Bigl(\frac{1}{K}\sum_{i=1}^K\mu_i,\mu\Bigr)<\frac{\eps}{8}, \label{pf.4.4-m1}
\end{equation}
and
\begin{equation}
    \Bigl\vert \frac{1}{K}\sum_{i=1}^K h_{\mu_i}(T)-h_\mu(T)\Bigr\vert<\frac{\eps}{16}. \label{pf.4.4-h1}
\end{equation}
For  for $i=1,2,\dotsc,K$, let
\begin{equation*}
    F_i=\Bigl\{\nu\in M_T(X)\colon \dbl(\mu_i,\nu)<\frac{\eps}{8}\Bigr\}.  
\end{equation*}
By Lemma~\ref{lem:sep}, there exist $\gamma>0$ and $N\in\N$ such that
for any $m>N$ there exists
\begin{equation}
    \Gamma_{m,i}\subset X_{m,F_i}\cap \supp(\mu_i) \label{pf.4.4-Gamma}
\end{equation}
which is $(m,\gamma)$-separated and satisfies
\begin{equation}
    \Bigl|\frac{1}{m}\log|\Gamma_{m,i}|-h_{\mu_i}(T)\Bigr|<\frac{\eps}{16}
\end{equation}
for $i=1,2,\dotsc,K$.
Removing sufficiently many elements from $\Gamma_{m,i}$
we may assume that
\begin{equation}
\Bigl|\frac{1}{K}\sum_{i=1}^K \frac{1}{m}\log |\Gamma_{m,i}| -c\Bigr|<\frac{\eps}{8}.\label{pf.4.4-h2}
\end{equation}

By \eqref{pf.4.4-Gamma}, for every $x\in \Gamma_{m,i}$,
\begin{equation}
    \dbl(\E_m(x),\mu_i)<\frac{\eps}{8}. \label{pf.4.4-m3}
\end{equation}
There is also $\beta>0$ such that $|\Gamma_{m,i}|\leq \beta^m$ for all $m>N$.

Let $\eta\in(0,\frac{\eps}{16})$, let $\delta\in(0,\frac{\gamma}{4})$ be provided for $\frac{\eta}{8}$
and let $\xi\in (0,\frac{\delta}{8})$ be provided for $\frac{\delta}{8}$ by the shadowing property.
Let $U_1,\ldots, U_p$ be an open cover of $\supp(\mu)$ by sets with $\diam (U_i)<\xi$.
As $\supp(\mu)/_\sim$ is a singleton,
there exists $M>0$ such that for any $U_i$ and $U_j$ there exists
$0<M_{i,j}\leq M$ with $\supp(\mu)\cap U_i\cap T^{-M_{ij}}U_j\neq\emptyset$.
When $m$ is fixed, for each $i=1,2,\dotsc,K$, using the pigeon-hole principle there are sets $V_i,W_i$
in the list $\{U_k\}_{k=1}^p$ such that
\begin{equation}
    |\Gamma_{m,i}\cap V_i\cap T^{-m}(W_i)|\geq   \frac{|\Gamma_{m,i}|}{p^2}.\label{4.x}
\end{equation}
Taking $m$ sufficiently large and putting $\Gamma'_{m,i}=\Gamma_{m,i}\cap V_i\cap T^{-m}(W_i)$,
we easily obtain by \eqref{4.x} that
\begin{equation}
    \Bigl|\frac{1}{m}\log|\Gamma'_{m,i}|-\frac{1}{m}\log|\Gamma_{m,i}|\Bigr|<\frac{\eps}{8}\label{pf.4.4-h3}
\end{equation}
for $i=1,2,\dotsc,K$. We may additionally assume that
\[
    \frac{(4+\log\beta)M}{m}<\frac{\eps}{8} \quad \textrm{ and }\quad M<m.
\]

Let $s_i=|\Gamma'_{m,i}|$ and enumerate  $\Gamma'_{m,i}$
as $\{y^{(i)}_0, y^{(i)}_1,\dotsc,y^{(i)}_{s_i-1}\}$.
For $i=1,2,\dotsc,K-1$, we pick $z_i\in W_i$ and $0<M_i\leq M$ with $T^{M_i}(z_i)\in V_{i+1}$.
We also pick $z_K\in W_K$ and $0<M_K\leq M$ with $T^{M_K}(z_K)\in V_1$.
Let us denote
\[
    n=Km+\sum_{i=1}^{K}M_{i}
\]
and observe that $Km<n\leq K(m+M)$.
For each choice of indexes $0\leq i_j <s_i$ for $i=1,\ldots, K$,
define a sequence of points in $X$, denoted by  $\{y_j\colon 0\leq j\leq n-1\}$,  as follows
\[y^{(1)}_{i_1},T(y^{(1)}_{i_1}),\dotsc,T^{m-1}(y^{(1)}_{i_1}),z_{1},Tz_{1},\dotsc,T^{M_{1}-1}z_{1},\]
\[y^{(2)}_{i_2},T(y^{(2)}_{i_2}),\dotsc,T^{m-1}(y^{(2)}_{i_2}),z_{2},Tz_{2},\dotsc,T^{M_{2}-1}z_{2},\]
\[...\quad ...\quad ...\]
\[y^{(K)}_{i_1},T(y^{(K)}_{i_2}),\dotsc,T^{m-1}(y^{(K)}_{i_K}),z_{K},Tz_{K},\dotsc,T^{M_{K}-1}z_{K}.\]

By Lemma~\ref{lem:measure-approx}\eqref{enum:measure-approx-1} we obtain that
\begin{equation}\label{pf.4.4-m4}
    \begin{split}
    \dbl\Bigl(\frac{1}{K}\sum_{i=1}^K\E_m(y^{(i)}_{i_1}),\frac{1}{n}\sum_{j=0}^{n-1}\delta_{y_j}\Bigr)    
    &\leq \frac{\Bigl(Km+\sum\limits_{i=1}^K(m+M_i)\Bigr)\sum\limits_{i=1}^K M_i+Km
    	\sum\limits_{i=1}^K M_i}{K m \sum\limits_{i=1}^K(m+M_i)}\\
    &\leq \frac{(2m+M)M+Mm}{m^2}\leq \frac{4M}{m}<\frac{\eps}{8}.
    \end{split}
\end{equation}
Let $\alpha$ be the periodic sequence
\[
    \alpha=(y_0,y_1,\dotsc,y_{n-1},y_0,y_1,\dotsc,y_{n-1},y_0,\dotsc).
\]
Then $\alpha$ is a $\xi$-pseudo-orbit.
Pick a point $y=y(i_1,i_2,\dotsc,i_K)$ which $\frac{\delta}{8}$-traces $\alpha$, that is
\[
  d(T^{jn+i}(y),y_i)<\frac{\delta}{8}
\]
for every $j\geq 0$ and $i=0,1,\dotsc,n-1$.
By Lemma~\ref{lem:measure-approx}\eqref{enum:measure-approx-2}, we get
\begin{equation}
  \dbl\Bigl(\E_n(y),\frac{1}{n}\sum_{j=0}^{n-1}\delta_{y_j}\Bigr)
  <\frac{\delta}{8}<\frac{\eps}{8}. \label{pf.4.4-m5}
\end{equation}
Let $s=\prod_{i=1}^K s_i$ and $\Gamma$ be the collection of all those $y(i_1,i_2,\dotsc,i_K)$.
Take any two distinct points $p,q\in \Gamma$, say $p=y(i_1,\ldots,i_K)$, $q=y(j_1,\ldots, j_K)$
	and $i_u\neq j_u$ for some $1\leq u \leq K$. Then
	$$
	d(T^{um+\sum_{i=1}^{u-1}M_i+t}(p),T^t(y^{(u)}_{i_u}))<\frac{\delta}{8}\quad , \quad 	d(T^{um+\sum_{i=1}^{u-1}M_i+t}(q),T^t(y^{(u)}_{j_u}))<\frac{\delta}{8}
	$$
	for each $t=0,1,\ldots,m-1$.
Recall that each $\Gamma'_{m,i}$ is $(m,\gamma)$-separated, hence there is $0\leq v < m$
such that $d(T^v(y^{(u)}_{i_u}),T^v(y^{(u)}_{j_u}))>\gamma$ and therefore
$$
	d(T^{um+\sum_{i=1}^{u-1}M_i+v}(p),T^{um+\sum_{i=1}^{u-1}M_i+v}(q))\geq  \gamma-\frac{\delta}{4}>\frac{\gamma}{2}.
$$
This proves that
$\Gamma$ is $(n,\frac{\gamma}{2})$-separated  
and clearly also $|\Gamma|=s$.
We also have that
\[
    \diam(\Gamma\cup T^n(\Gamma))<\diam(V_1)+2\cdot\frac{\delta}{8}<\delta.
\]
Applying Lemma~\ref{lem:3} to $\Gamma$, $n$, $\frac{\gamma}{2}$, $\delta$ and $\eta$,
there exists a regularly recurrent point $z\in X$ such that
for every $j\geq 0$ there is $q_j\in\Gamma$ such that
\begin{equation}
    d(T^{jn+t}(z),T^t(q_j))\leq2\eta<\frac{\eps}{8}\label{4.c*}
\end{equation}
for $t=0,\ldots,n-1$ and
\begin{equation}
    (1-\eta)\frac{\log(s)}{n} \leq \htop (\overline{\Orb(z,T)},T) \leq \frac{\log(s)}{n}. \label{pf.4.4-h4}
\end{equation}
By Lemma~\ref{lem:measure-approx}\eqref{enum:measure-approx-2} and \eqref{4.c*},
for every $j\geq 0$ there is $q_j\in\Gamma$ such that
\begin{equation}
    \dbl(\E_{n}(T^{jn}(z)), \E_{n}(q_j))<\frac{\eps}{8}, \label{pf.4.4-m6}
\end{equation}
By the variational principle of the topological entropy,
there exists an ergodic measure $\nu$ supported on $\overline{\Orb(z,T)}$  such that
\begin{equation}
    (1-2\eta)\frac{\log(s)}{n}<h_\nu(T)\leq \htop (\overline{\Orb(z,T)},T)\leq \frac{\log(s)}{n}. \label{pf.4.4-h5}
\end{equation}
For sufficiently small positive number $\eta$ we have
\begin{equation}
    2\eta\frac{\log(s)}{n}=\Bigl|(1-2\eta)\frac{\log(s)}{n}
        -\frac{1}{n}\sum_{i=1}^k \log |\Gamma'_{m,i}|\Bigr|<\frac{\eps}{8}. \label{pf.4.4-h6}
\end{equation}
On the other hand
\begin{equation}
    \Bigl|\frac{1}{n} \log |\Gamma'_{m,i}|-\frac{1}{Km}\sum_{i=1}^k \log |\Gamma'_{m,i}|\Bigr|
        \leq \frac{M}{n}\log \beta<\frac{\eps}{8}.  \label{pf.4.4-h7}
\end{equation}
Summing up formulas~\eqref{pf.4.4-h1}, \eqref{pf.4.4-h2}, \eqref{pf.4.4-h3},
\eqref{pf.4.4-h4}, \eqref{pf.4.4-h5}, \eqref{pf.4.4-h6} and \eqref{pf.4.4-h7}, we get
\begin{equation}
    |h_\nu(T)-c|<\eps.\label{4.y}
\end{equation}
By Lemma~\ref{lem:measure-in-the-orbit}
there exist $j_1,j_2\in\N$ such that
\begin{equation}
  \dbl(\E_{j_1n}(T^{j_2n}(z)),\nu)<\frac{\eps}{8}. \label{pf.4.4-m7}
\end{equation}
Summing up formulas~\eqref{pf.4.4-m1}, \eqref{pf.4.4-m3}, \eqref{pf.4.4-m4}, \eqref{pf.4.4-m5}
and \eqref{pf.4.4-m6} and Lemma~\ref{lem:measure-approx}\eqref{enum:measure-approx-3}, we get
\[
    \dbl(\mu,\nu)<\eps.
\]
The proof of this case is finished.

\noindent \textbf{Case 3:} $h_\mu(T)=\infty$. This case is similar to the Case 2,
with the main change in \eqref{pf.4.4-h1}, which by Lemma~\ref{lem:mes:ent:approx:2}
has the form
\begin{equation}
    \frac{1}{K}\sum_{i=1}^K h_{\mu_i}(T)>\frac{8}{\eps}.
\end{equation}
All the other estimates have to be adjusted accordingly, leading to $h_\nu(T)>1/\eps$ in \eqref{4.y}
in the case of $c=\infty$.
We leave the details of the proof in this case to the reader.
\end{proof}

Now we are ready to prove Corollary C.
\begin{proof}[Proof of Corollary C]
(1) Fix any $0\leq c<\htop(T)$ and $\eps>0$.
Let $\mu$ be an invariant measure with $h_\mu(T)\geq c$ and $\eta>0$.
If $h_\mu(T)>c$, then by Theorem B
there exists an ergodic measure $\nu$ supported on an almost 1-1 extension of odometer such that
$\dbl(\mu,\nu)<\eta$ and $c<h_\nu(T)<c+\eps$.
If $h_\mu(T)=c$, then by the variational principle of the topological entropy, pick
an invariant measure $\mu'$ such that $c<h_{\mu'}(T)\leq \htop(T)$,
and next pick a sufficiently small number $t\in (0,1)$ such that $\dbl(\mu, \mu'')<\frac{\eta}{2}$,
where $\mu'' = (1-t)\mu + t\mu'$.
Observe that
$h_{\mu''}(T)= (1-t)h_\mu(T)+t h_{\mu'}(T)>c$.
By Theorem B again,
there exists an ergodic measure $\nu$ support on an almost 1-1 extensions of odometer such that
$\dbl(\mu'',\nu)<\frac{\eta}{2}$ and $c<h_\nu(T)<c+\eps$.
So $\dbl(\mu,\nu)<\eta$.

(2)
Now assume that the entropy function is upper semi-continuous.
Fix any $0\leq c<\htop(T)$ and denote
\[\mathcal{V}=\{\mu\in M_T^{erg}(X)\colon h_\mu(T)\geq c\}\ \text{ and }\
\mathcal{V}(c)=\{\mu\in M_T(X)\colon h_\mu(T)=c\}.\]
For every $\eps>0$, put
$\mathcal{V}(c,\eps)=\{\mu\in M_T(X)\colon c\leq h_\mu(T)<c+\eps\}$.
By the first part of the conclusion we know that for every $\eps>0$,
$\mathcal{V}(c,\eps)$ is dense in $\mathcal{V}$.
As the entropy function is upper semi-continuous, $\mathcal{V}(c,\eps)$ is open in $\mathcal{V}$.
Then $\mathcal{V}(c)$ is residual in $\mathcal{V}$,
because $\mathcal{V}(c)=\bigcap_{n=1}^\infty \mathcal{V}(c,\tfrac{1}{n})$.
\end{proof}

\begin{rem}
Note that Corollary C does not contain the case $c=\htop(X,T)$.
\end{rem}

As mentioned in the introduction,
if an expansive dynamical system has the periodic specification property
(or its weaker version, e.g. approximate product property \cite{PS})
then each invariant measure is entropy-approachable by ergodic measures.
It is natural to ask that whether the expansiveness is redundant.
More precisely, we have the following question.

\begin{ques}\label{que:app}
Assume that $(X,T)$ has approximate product property (or  periodic specification property).
Is it true that for every invariant measure $\mu$ there are ergodic measures $\nu_n\in M_T(X)$
such that  $\lim_{n\to \infty}\nu_n=\mu$ and $\lim_{n\to \infty}h_{\nu_n}(T)=h_\mu(T)$?
\end{ques}

\section{Examples of shadowing beyond specification property}\label{sec:5}

When $(X,T)$ is a topologically mixing map with the shadowing property (and the entropy function is upper semi-continuous), Theorems~A and B do not lead us too much beyond what is already known,
because results for dynamical systems with the specification property can be applied. Still, we can say more about the topological
structure of supports of ergodic measures, but results on approximation of entropy follow by earlier results.
The aim of this section is two fold. First, we will provide natural examples without specification property. These examples will also show that not much more can be said about dynamics on supports of approximating measures.
Second, we will be able to highlight possible problems that may arise when applying Theorem~\ref{thm:chain-rec:ent}
 to measures supported on chain-recurrent classes, that is when there is lack of transitivity.

Consider the family of tent maps $f_\lambda\colon [0,1]\to [0,1]$, with $\lambda\in [\sqrt{2},2]$
given by the formula:
$$
f_\lambda(x)=\begin{cases}
\lambda x,& 0\leq x\leq 1/2\\
\lambda(1-x),& 1/2\leq x \leq 1
\end{cases}.
$$
By Theorem~6.1 in \cite{Yorke} the set  $\mathcal{S}\subset [\sqrt{2},2]$ for which $f_\lambda$ has the shadowing property is of full Lebesgue
measure (however has complement uncountable in any open subset of $[\sqrt{2},2]$). It is also known that $2\in \mathcal{S}$
but not much more can be said about exact structure of the set $\mathcal{S}$.

It is not hard to see that $\Omega(T_\lambda)=\{0\}\cup C_\lambda$ where $C_\lambda$ is the \textit{core} of tent map defined by $C_\lambda=[f_\lambda^2(1/2),f_\lambda(1/2)]$. Since any surjective map with the shadowing property restricted to its non-wandering set also has the shadowing property (e.g. see \cite[Theorem~3.4.2.]{AH}), the tent map restricted to its core $T_\lambda=f_\lambda\colon C_\lambda\to C_\lambda$  has the shadowing property.
 Clearly, after renormalization we can view $T_\lambda$ as a map on $[0,1]$, and then it is given by the formula
 $$
 T_\lambda(x)=\begin{cases}
 \lambda x+(2-\lambda ),& 0\leq x\leq \frac{\lambda-1}{\lambda}\\
 -\lambda x+\lambda,& \frac{\lambda-1}{\lambda}\leq x \leq 1
 \end{cases}.
 $$
Note that each $T_\lambda$ is transitive (e.g. see \cite[Remark~3.4.17]{BB}). But it is also easy to see that $T_\lambda([0,p])\neq [p,1]$, where $p$ is the unique fixed point of $T_\lambda$, hence $T_\lambda$ is topologically mixing (e.g. see \cite[Theorem~2.1]{HKO}).

\begin{exam}\label{ex:shad:nospec}
Let $Y=[0,1]\times G_\mathbf{s}$ for some odometer $(G_\mathbf{s},T_{\mathbf{s}})$ with $|G_\mathbf{s}|>1$ and let $F=T_\lambda \times T_\mathbf{s}$, for some $\lambda\in \mathcal{S}$. Then $(Y,F)$ is a transitive dynamical system and it has shadowing property as
a product of maps with the shadowing property. But $(Y,F)$ is not mixing, as it has a nontrivial equicontinuous factor $(G_\mathbf{s},T_{\mathbf{s}})$, in particular it does not have specification property.
\end{exam}
\begin{rem}
Note that any minimal subsystem of $(Y,F)$ in Example~\ref{ex:shad:nospec} has odometer $(G_\mathbf{s},T_{\mathbf{s}})$ as its factor, hence in Theorems~A  and B we cannot replace odometers by other types of minimal systems (e.g. weakly mixing).
\end{rem}

\begin{rem}
It is clear that in construction of $(Y,F)$ in Example~\ref{ex:shad:nospec} the only essential ingredient is that $([0,1],T_\lambda)$
	has the shadowing property, so we can replace it by any system $(X,T)$ with the shadowing property and the same argument will work. However we will need special structure of Example~\ref{ex:shad:nospec} to construct a map on the unit square later in this section.
\end{rem}

Now, we will provide another example (this time on the unit square), which will give us a better insight into statements of Theorem~\ref{thm:chain-rec:ent}.

First, consider the map $F_D$ of type $2^\infty$ (i.e. map with periodic points of primary period $2^n$ for each $n$ and no other periods), introduced by Delahaye in \cite{Dela} (see also \cite{RuetteSurvey}).
This map is defined on $[0,1]$ by the following rules (its graph is presented on Figure~\ref{pic:f}):
\begin{itemize}
	\item $F_D(0)=\frac{2}{3},f(1)=0$,
	\item $F_D(1-\frac{2}{3^n})=\frac{1}{3^{n-1}}$, and $F_D(1-\frac{1}{3^n})=\frac{2}{3^{n+1}}$ for all $n\geq 1$,
	\item $F_D$ is linear between the above points.
\end{itemize}

\begin{figure}[!htbp]
	\begin{center}
		\includegraphics[width=0.4\textwidth]{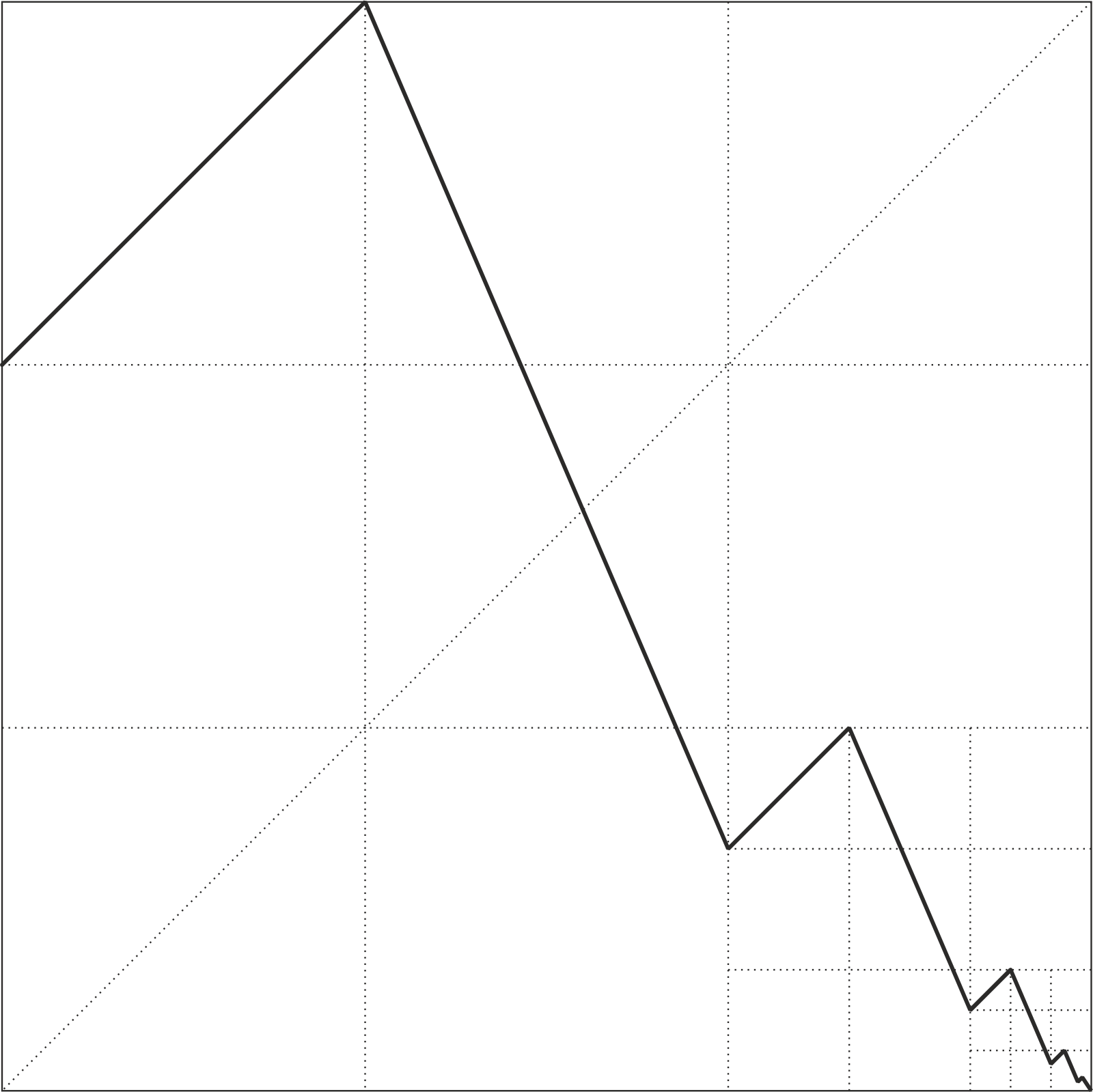}
		\includegraphics[width=0.4\textwidth]{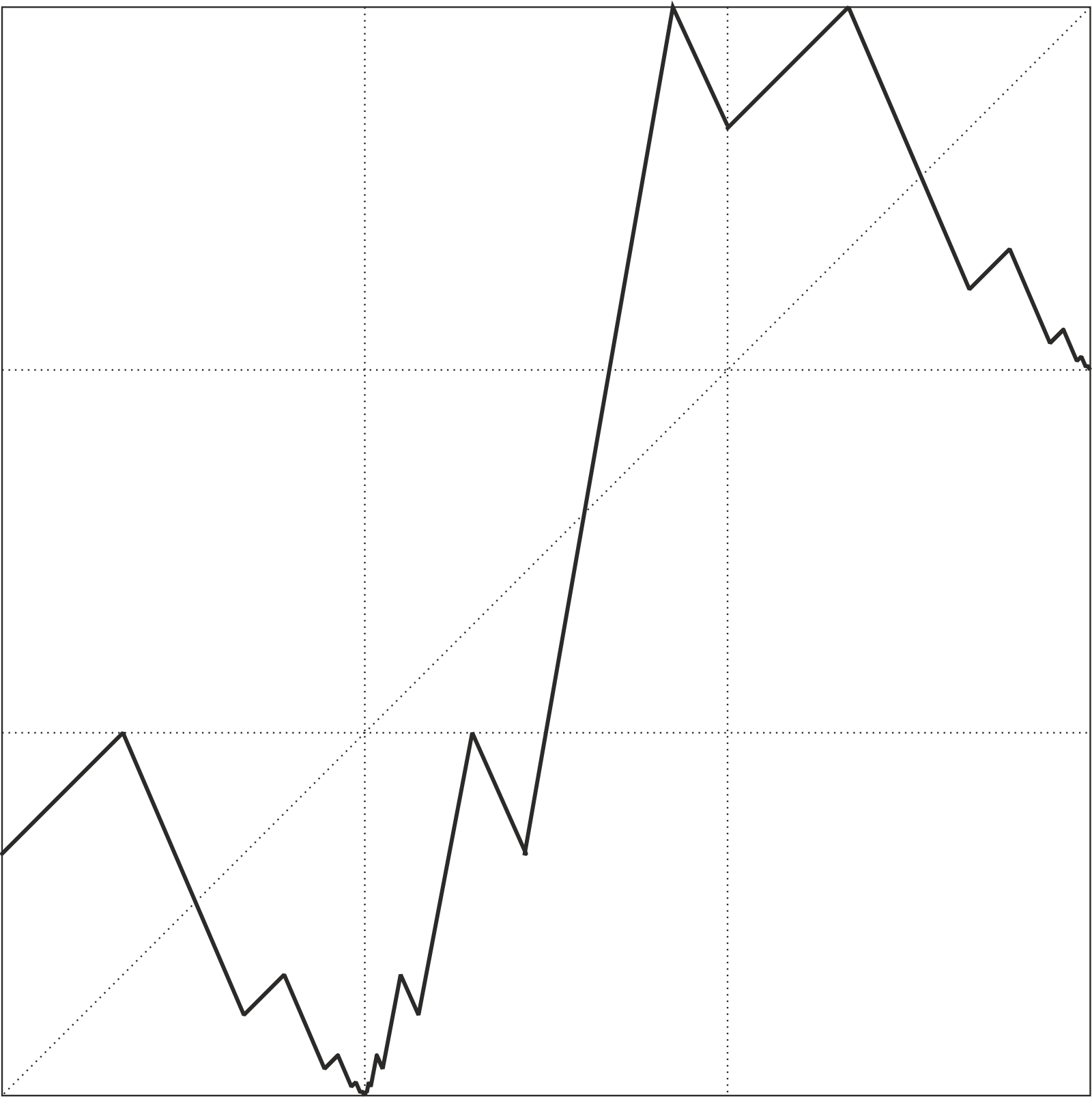}
		\begin{minipage}[c]{0.8\textwidth}
			\begin{center}
				\caption{Graph of Delahaye map $F_D$ and its second iterate $F_D^2$}\label{pic:f}
			\end{center}
		\end{minipage}
	\end{center}
\end{figure}
If we analyze the graph of $F_D^2$ then the following is evident:
\begin{enumerate}
\item\label{eq:TD:p1} Both $F^2_D|_{[0,1/3]}$ and $F^2_D|_{[2/3,1]}$ after renormalization became again $F_D$,
\item Set $(1/2, 2/3)$ contains a unique repelling fixed point $p$ such that:
\begin{enumerate}
\item For every $x\in (1/3,2/3)$, $x\neq p$ there exists $n$ such that $F_D^n(x)\not \in (1/3,2/3)$.
\item $F_D^{-1}(\{p\})=\{p\}$.
\end{enumerate}
In particular, if $\liminf_{n\to\infty} |F_D^n(x)-p|=0$ then $x=p$.
\end{enumerate}
It is also clear that using \eqref{eq:TD:p1}, properties of the fixed point $p$ above can be restated
for any other periodic point of $F_D$ in the analogous manner.
Using  \eqref{eq:TD:p1} it can be also proved that $F_D$ contains an infinite $\omega$-limit set conjugated to an odometer (e.g. see \cite{RuetteSurvey}).
It is a unique infinite $\omega$-limit set of $F_D$.

Before we can explain why $F_D$ has the shadowing property, we will need three definitions form \cite{Kuchta}.
If $a,b\in [0,1]$ and $a\neq b$ then we write $\seq{a,b}$ to denote interval spanned by $a$ and $b$.
\begin{defn}
Let $p$ be a periodic point of a map $f\colon [0,1]\to [0,1]$ and let $q\in [0,1]$, $q\neq p$.
We say that one-sided neighborhood $\seq{p,q}$ of $p$ is $f$-$m$-nontrapping if $f^m(p)=p$ and
for every $x\in \seq{p,q}$ we have $x\in f^m(\seq{p,q})$.
\end{defn}

\begin{defn}\label{def:ntr}
A continuous function $f\colon [0,1]\to [0,1]$ is \textit{nondegenerate} if the following condition holds:

If $x\in [0,1]$, $p$ is a periodic point of $f$, $\seq{p,q}$ is an $m$-$f$-nontrapping neighborhood of $p$ and $\lim_{n\to\infty} f^{nm}(x)=p$,
then for every open neighborhood $O_x$ of $x$ and for all $z_1,z_2\in \seq{p,q}\setminus \{p,q\}$ there is an $n\geq 0$ such that $\seq{z_1,z_2}\subset f^{mn}(O_x)$.
\end{defn}

Consider again the fixed point $p$ of $F_D$. The only candidate for $x$ in Definition~\ref{def:ntr} is $x=p$. But it is clear (see Figure~\ref{pic:f}) that for every open interval $J\ni p$
and every $m$ there is $k>0$ such that $F_D^{mk}(J)\supset [1/3,2/3]$ and therefore
$$
F_D^{2mk}(J)=F_D^{mk-2}(F_D^2(F_D^{mk}(J)))=F_D^{mk-2}(F_D^2([1/3,2/3])=F_D^{mk-2}([0,1])=[0,1].
$$
It is also clear that if $u\neq p$ is periodic point of $F_D$, then for any $F_D$-$m$-nontrapping neighborhood $\seq{u,q}$ of $u$ we must have $\seq{u,q}\cap (1/3,2/3)=\emptyset$ (see graph of $F_D^2$ on Figure~\ref{pic:f}).
Therefore by renormalization provided by \eqref{eq:TD:p1}, verification of conditions in Definition~\ref{def:ntr} can be reduced to the case of $p$ and $F_D$. This shows that $F_D$ is a nondegenerate map.

An interval $J$ is \textit{periodic}, if there exists $k>0$ such that $f^k(J)=J$. We denote by $\Per(J)$
the least such integer $k$.

\begin{defn}
We say that a continuous function $f\colon [0,1]\to [0,1]$ is a \textit{shrink function} if for every sequence $\{J_k\}_{k=1}^\infty$ of periodic intervals such that $J_{k+1}\subset J_k$ and $\Per(J_k)<\Per(J_{k+1})$
we have that $\lim_{k\to \infty} \diam (J_k)=0$.
\end{defn}
Again by \eqref{eq:TD:p1}, it is not hard to see that $F_D$ is a shrink function. Therefore $F_D$ has the shadowing property, by the Main Theorem from \cite{Kuchta}, which we present below.

\begin{thm}
Let $f\colon [0,1]\to [0,1]$ be a continuous function with zero topological entropy. Then the following conditions are equivalent:
\begin{enumerate}
\item $([0,1],f)$ has the shadowing property,
\item $f$ is a nondegenerate shrink function.
\end{enumerate}
\end{thm}

\begin{exam}\label{ex:2}
Let $X=[0,1]^2$ and let $T=F_D\times T_\lambda$ for some $\lambda\in \mathcal{S}$. Then $(X,T)$ is a dynamical system with the shadowing property.
We examine chain recurrent sets of $(X,T)$
If we denote $Q=[x]_\sim$ for some $x\in CR(X,T)$ then $F_D|_Q$
can be presented as a product $T_\mathbf{s}\times T_\lambda$ for some odometer $(G_{\mathbf{s}},T_\mathbf{s})$, where $|G_{\mathbf{s}}|=2^n$ for some $n\geq 0$
or $|G_{\mathbf{s}}|=+\infty$.
Thus by Theorem~B we know that each $\{\mu \in M_T(X) : \supp(\mu) \subset [x]_\sim\}$ is a Paulsen simplex, because each $F_D|_Q$ has the shadowing property (see Example~\ref{ex:shad:nospec}).

By Theorem~\eqref{thm:chain-rec:ent} we can approximate every invariant measure $\mu$ with $\supp(\mu) \subset Q$ by an ergodic measure supported on an almost 1-1 extension of an odometer. But note that infinite $\omega$-limit set $A$ for $F_D$ can be approximated (in the sense of Hausdorff metric) arbitrarily close by periodic orbits
and clearly measures supported on these orbits approximate Haar measure on $A$. Then there are plenty of ergodic measures for $T$ with supports outside $A\times [0,1]$ which can approximate any
 invariant measure supported on $A\times [0,1]$.
In other words, we may not hope that any of ergodic measures $\mu_n$ provided by Theorem~\eqref{thm:chain-rec:ent}
intersects the chain-recurrent class defined by $\supp(\mu)$. In fact supports of measures $\supp(\mu_n)$ for distinct $n$ can belong to pairwise different classes of the chain-recurrent relation $\sim$.
\end{exam}

\begin{rem}
While there are numerous smooth maps of type $2^\infty$ on $[0,1]$, they do not have obvious regular structure on graph of higher iterates, similar to that in Delahaye example (see Figure~\ref{pic:f}). While we believe that there exists a smooth example on a $2$-dimensional manifold with the structure of invariant measures (and chain recurrent sets) similar to that in Example~\ref{ex:2}, we do not see immediate construction of such a map.
\end{rem}

\section{Shadowing and measures of maximal entropy}\label{sec:6}
It was proved in~\cite[Theorem~2.1]{PS} that Question~\ref{que:app} has a positive answer
if we assume some weak form of expansivity (e.g. asymptotically $h$-expansive),
because in this case  the entropy function is upper semi-continuous.
It is also not hard to verify that if the entropy function is upper semi-continuous
then there exists an ergodic measure $\nu$ with maximal entropy, that is $h_\nu(T)=\htop(T)$.
The aim of this section is the construction of a dynamical system which
is transitive and has the shadowing property, but without invariant measure with maximal entropy, proving that way Theorem~D.
This shows that the case of non-expansive dynamical systems with the shadowing property is more delicate than the hyperbolic case.

Before we can start the proof, we will need a few facts on shifts on infinite alphabets
and their representations. Fix some positive integer $N$ and denote $P=2^N+1$.
We will define an infinite graph $G$ on a countable set of vertexes
\begin{equation}\label{def:VG}
    V(G)=\{u\}\cup \bigcup_{n=1}^\infty \{ v_k^{n,i} : 1\leq i \leq a(n), 1\leq k <n\}\cup \bigcup_{n=P}^\infty\{w_k^n : 1\leq k< n\}
\end{equation}
where $a(n)$ is the following sequence:
$a(1)=1$,  $a(2^N)=2^{2^N-N}-1$,  $a(2^k)=2^{2^k-k}$ for $k\geq 2$ and $k\neq N$,
 and $a(n)=0$ for all other arguments.

For technical reasons denote $v_0^{n,i}=v_n^{n,i}=u$ for $i=1,\ldots, a(n)$, and $w_0^n=w_n^n=u$.
We add the following edges in $G$:
\begin{enumerate}
\item there is edge $v_{k-1}^{n,i}\to v_{k}^{n,i}$ for $k=1,\ldots,n$,
$i=1,\ldots,a(n)$ and every $n\geq 1$,
\item there is edge $w_{k-1}^{n}\to w_{k}^{n}$ for $k=1,\ldots,n-1$ and every $n\geq P$.
\end{enumerate}
We denote by $\Gamma_G$ the set of bi-infinite paths on $G$, that is
\begin{equation}\label{def:GG}
    \Gamma_G=\{(v_n)_{n\in \Z} : v_n\to v_{n+1} \text{ in }G\text{ for every }n\in\Z\}
    \subset (V(G))^{\mathbb{Z}}.
\end{equation}
We endow $\Gamma_G$ with standard action of left shift $\sigma(x)_i=x_{i+1}$ for all $i\in \Z$.
We refer the  reader to \cite{Kit} for more details on countable Markov shifts.

As usual, we say that a sequence of vertexes $u_0,\ldots, u_n\in V(G)$ is a \emph{path} on $G$,
if for every $i=1,\ldots,n$ there is an edge $u_{i-1}\to u_i$.
Path with additional condition $u_0=u_n$ is a \emph{cycle}.
We say that the graph $G$ is \emph{strongly connected}
if there is a path between any two vertexes $u,v\in V(G)$.
Enumerate elements in $V(G)=(q_n)_{n=1}^\infty$ in such a way that:
\begin{enumerate}
\item $q_1=u$,
\item if $q_i=v_k^{n,r}$ and $q_j=v_{l}^{m,t}$ and $n<m$ then $i<j$,
\item if $q_i=w_k^{n}$ and $q_j=w_{l}^{m}$ and $n<m$ then $i<j$.
\end{enumerate}
If we identify $q_n=1/n\in [0,1]$ then putting $q_0=0$
we obtain one point compactification of $(q_n)_{n=0}^\infty$.
Then we may view $\Gamma_G$ as a subset of $\mathbb{X}=[0,1]^\Z$ endowed with the metric
$$
d(x,y)=\sum_{i\in \Z}4^{-|i|}|x_i-y_i|.
$$
Note that $(\mathbb{X},d)$ is a compact metric space.

Furthermore, taking the closure of $\Gamma_G$ in $\{q_n : n=0,1,\ldots\}^\Z$
we see that $\overline{\Gamma_G}=\Gamma_G\cup Q$ where $Q$ contains
$0^\infty=(\ldots000\ldots)$, points $\ldots 00q$ and $p00\ldots$
where $q$ (resp.\ $p$) is infinite to the right (resp.\ to the left) path on $G$ starting at vertex $u$
and points of the form $\ldots 00r00\ldots$ where $r$ is a finite path on $G$ starting and ending at $u$.
Note that $p00\ldots$ and $\ldots 00r00\ldots$ are asymptotic to $0^\infty$
and $\ldots 00q$ is asymptotic to a point completely contained in $\Gamma_G$.
This implies that the only $\sigma$-invariant measure
which was not visible on $\Gamma_G$ is supported on the point $0^\infty$.
Then ergodic measures with positive entropy on $(\Gamma_G,\sigma)$ coincide with
ergodic measures with positive entropy supported for $(\overline{\Gamma_G},\sigma)$. 

Let $p_{uv}^G(n)$ be the number of paths of length $n$ between vertexes $u$ and $v$ on $G$ and
let $R_{uv}(G)$ be the radius of convergence of the series $\sum_{n=1}^\infty p_{uv}(n)z^n$.
The following results was first proved by Vere-Jones \cite{VJ}.

\begin{lem}\label{R:VJ}
	Let $G$ be a strongly connected graph. Then $R_{uv}(G)$ does not depend on the choice of $u,v$.
\end{lem}

The unique number $R_{uv}(G)$ provided by Lemma~\ref{R:VJ} will be simply denoted by $R$.
Vere-Jones proved \cite{VJ} that always $R\leq 1$.

Denote by $f_{ww}^G(n)$ the number of indecomposable cycles from $w$ to $w$,
that is paths $u_0\to u_1\to \dotsb\to u_n$ in $G$ with $u_0=u_n=w$ and $u_i\neq w$ for every $0<i<n$.
Let $L_{ww}(G)$ be the radius of convergence of the series $\sum_{n=1}^\infty f_{ww}(n)z^n$.
We present the following result of Salama \cite{Sal} (see also \cite{Rue}).
\begin{lem}\label{lem:Rue}
	If $G$ is null recurrent then $R=L_{ww}$ for all vertexes $w\in V(G)$.
\end{lem}

By our construction
we have $f_{uu}^G(n)=a(n)$ for $n<P$ and $f_{uu}^G(n)=a(n)+1$ for $n\geq P$.
The radius of convergence of the series
$\sum_{n=1}^\infty f_{uu}^G(n)z^n$ is $L_{uu}=1/2$.
Furthermore
\begin{align*}
\sum_{n=1}^\infty f_{u,u}^G(n) L^n&= \frac{1}{2}-2^{-P+1}+\sum_{n=2}^\infty  2^{2^n-n}
\cdot 2^{-2^n}+\sum_{n=P}^\infty 2^{-n}=1,\\
	\sum_{n=1}^\infty n f_{u,u}^G(n) L^n&\geq  \sum_{n=P}^\infty  2^n\cdot 2^{2^n-n}\cdot 2^{-2^n}=\sum_{n=P}^\infty 1=+\infty.
\end{align*}
Using terminology of Vere-Jones \cite{VJ}, it means that $G$ is a null-recurrent graph.
By Lemma~\ref{lem:Rue} we obtain that $R=1/2$.

Recall that \emph{Gurevich entropy} of a graph $G$ is the number
\[
    h(G)=\sup \{\htop(\sigma|_{\Gamma_H}) : H \text{ is a finite subgraph of }G\}.
\]
The following result of Gurevich \cite{Gur}
will allow us to calculate entropy of $(\overline{\Gamma_G},\sigma)$ (see also \cite{Rue}).

\begin{thm}
Let $G$ be a strongly connected graph with $R>0$. Then
\begin{align*}
    \htop(\overline{\Gamma_G},\sigma)&=h(G)=-\log R\\
	  &=\sup\{h_\mu(\sigma) : \mu \text{ is a }\sigma\text{-invariant measure for }(\Gamma_G,\sigma)\}.
\end{align*}
\end{thm}

This shows that in our construction the Gurevich entropy of $G$ is finite and equals $\log 2$.
Now we can apply another result of Gurevich \cite{Gur} (see also \cite{Kit}).
\begin{thm}\label{thm:Gur}
Let $G$ be a strongly connected graph with positive finite entropy. If $G$ is null recurrent
then $(\overline{\Gamma_G},\sigma)$ does not have measure of maximal entropy.
\end{thm}

Now we have enough tools to start the proof.
\begin{proof}[Proof of Theorem~D]
Take as $(X,T)$ the system $(\overline{\Gamma_G},\sigma)$ defined by \eqref{def:VG} and \eqref{def:GG}.
Transitivity follows directly by the construction (the graph $G$ is strongly connected) and
lack of measure of maximal entropy is a consequence of Theorem~\ref{thm:Gur}.
It remains to prove that $(\overline{\Gamma_G},\sigma)$ has the shadowing property.

Denote by $G_n$ the subgraph of $G$ with vertexes
\[
V(G_n)=\{u\}\cup \{v_k^{m,i} : 1\leq m\leq 2^n, 1\leq i \leq a(m), 1\leq k < n\}
\cup \{w_k^m : 1\leq k<m, P\leq m\leq n\}
\]
and all edges between $p,q\in V(G_n)$ contained in $G$. Then $\Gamma_{G_n}$ is a vertex shift,
which is conjugate to a shift of finite type (see \cite{LM}).
Therefore, by the result of Walters \cite{Walters},
each $(\Gamma_{G_n},\sigma)$ has the shadowing property.

Fix any $\eps>0$, and let $K\in\mathbb{N}$ be such that
if $p\in V(G)\setminus V(G_{K-1})$ then $d(p,q_0)=|p|<\eps/8$.
Let $J=\max\{P,K+1\}$ and observe that  $w^n_k\not \in V(G_K)$ if and only if $n\geq J$.
Take an auxiliary symbol $c$ and define a shift of finite type $Z$ in the following way.
Words forbidden for shift $\Gamma_{G_K}$ are forbidden in language of $Z$.
Words $vc$ and $cv$ are forbidden provided that $v\neq u$. Also words $uc^{k-1} u$ are forbidden
when $k<J$. Note that $\Gamma_{G_K}\subset Z$.

Let $\delta>0$ be provided by shadowing of $(Z,\sigma)$ to $\eps/8$.
We also assume that $\delta<\eps/8$ is sufficiently small, so that if $p,q\in V(G_{K})$ and $|p-q|<\delta$
then $p=q$ and $\dist (V(G_K), V(G)\setminus V(G_K))>\delta$.
Take any integer $M>J$ such that $\sum_{|i|\geq M}2^{-i}<\delta/8$.
There is $\gamma>0$ such that
if $x,y\in \overline{\Gamma_G}$ and $d(x,y)<\gamma$ then $|x_i-y_i|<\delta$ for each index $|i|\leq 2M+1$.
Let $(z^i)_{i\in \Z}$ be a $\gamma$-pseudo orbit in $\overline{\Gamma_G}$.
If we fix any $-M-1\leq j\leq M$ and any $i\in \Z$,
then $z^{i+1}_j=z^i_{j+1}$ or $\max\{z^i_{j+1},z^{i+1}_j\}<\eps/8$.

We are going to use points $z^i$ to construct points $y^i$ in $Z$ which form a $\delta$-pseudo orbit for $(Z,\sigma)$.
Fix any $i\in \Z$ and assume that $s<t$ are consecutive indexes such that
$z^i_s=z^i_t=u$, i.e. $z^i_j\neq u$ for $j\in (s,t)$.
If $z^i_j\in V(G_K)$ for some (thus all) $j\in (s,t)$ then we put $y^i_j=z^i_j$ for each $j\in (s,t)$.
If $z^i_j\not\in V(G_k)$ then $t-s>J$ and hence word $u c^{t-s-1} u$ is admissible in $Z$.
Then we put elements of this cycle as $y^i_s\ldots y^i_t=uc^{t-s-1}u$.
If $z^i_s=u$ and $z^i_j\neq u$ for each $j>s$ then
we put
$y_s^i y_{s+1}^{i}\ldots=ucc\ldots$ and note that in that case
$z_j^i<\eps/8$ for all $j>s$.
Similarly,
we put $\ldots y_{s-1}^i y_s^i=\ldots cc u$ when $z_s^i=u$ and $z_j^i\neq u$ for $j<s$.
Finally, if $z_j^i\neq u$ for every $j\in \N$, we put	
$y_j^i=c$ for each $j\in \Z$.

We have to show that the sequence constructed this way is a $\delta$-pseudo orbit.
Note that when $|j|\leq M$ then we either have $y^i_{j+1}= y^{i+1}_j=c$ or
$z^i_{j+1}=y^i_{j+1}= y^{i+1}_j=z^{i+1}_j$ and $z^{i+1}_j,z^{i}_{j+1}\in V(G_K)$.
We obtain that $d(\sigma(y_i), y_{i+1})\leq \sum_{|j|>M}2^{-j}<\delta$
so indeed $(y^i)_{i\in \Z}$ is a $\delta$-pseudo orbit.
Let $x\in Z$ be a point which $\eps/8$-traces the $\delta$-pseudo orbit $(y^i)_{i\in \Z}$.
Note that if $x_n\in V(G_K)$ then $x_n=y^n_0=z^n_0$ and when $x_n=c$ then $z^n_0<\eps/8$.
Replace $x$ by a point $q\in \overline{\Gamma_G}$ in the following way.
If $x_n\in V(G_K)$ then we keep $q_n=x_n$. If $x_{[i,j]}=uc^{k-1} u$ for some $k>0$ and $i<j$
then $k\geq J$ and then we put $q_{[i,j]}=u, w_1^{k},\ldots, w_{k-1}^{k}u$.
If $x_{[0,\infty)}=u c^\infty$ then we put $q_{[0,\infty)}=u 0^\infty$ and similarly we put $q_{(\infty,0]}=0^\infty u$ when $x_{(\infty,0]}=c^\infty u$.
Finally, when $x=\ldots ccc\ldots$ then we put $q=\ldots000\ldots$.

Observe that the point $q$ is $\eps$-tracing the pseudo orbit $(z^i)_{i\in \Z}$ because
on all positions on which $z^i_j\neq y_j^i$ with $|j|< M$ we have $z^i_j<\eps/8$ and so
\begin{align*}
d(\sigma^n(q), z^n)&\leq \sum_{|i|\geq M}4^{-i}|q_{n+i}-z^n_i|+
\sum_{y^n_i\neq c, i\in \Z}4^{-i}|q_{n+i}-z^n_i|+\sum_{y^n_i= c, |i|<M}4^{-i}\frac{\eps}{8}\\
&\leq \frac{\eps}{4}+d(\sigma^n(x),y^n)+\frac{\eps}{3}\leq \frac{2\eps}{3}+\frac{\eps}{8}<\eps.
\end{align*}
This completes the proof.
\end{proof}

\section*{Acknowledgments}
J. Li was supported in part by NSF of China (11401362, 11471125)
and P. Oprocha was partly supported by the project ``LQ1602 IT4Innovations excellence in science'' and AGH local grant.
The authors would like to thank Wen Huang and Dominik Kwietniak for numerous discussions on the topics covered by the paper.


\begin{thebibliography}{99}

\bibitem{AH} N.~Aoki and K.~Hiraide, \emph{Topological theory of dynamical systems}. Recent advances.
North-Holland Mathematical Library, 52. North-Holland Publishing Co., Amsterdam, 1994.

\bibitem{B72} R.~Bowen, \emph{Entropy-expansive maps}. Trans. Amer. Math. Soc.
\textbf{164} (1972), 323--331.

\bibitem{BowSpec} R.~Bowen, \emph{Periodic points and measures for Axiom A diffeomorphisms},
Trans. Amer. Math. Soc., \textbf{154} (1971), 377--397.

\bibitem{BowPOTP} R. Bowen, \emph{$\omega$-limit sets for Axiom A diffeomorphisms},
J. Differential Equations, \textbf{18} (1975), 333--339.

\bibitem{BB}  K. M. Brucks, H. Bruin, \emph{Topics from one-dimensional dynamics}, London Mathematical Society Student Texts, vol. 62. Cambridge University Press, Cambridge, 2004.

\bibitem{C15} H.~Comman, \emph{Criteria for the density of the graph of the entropy map
restricted to ergodic states}, to appear in Ergodic Theory and Dynamical Systems,
\href{http://arxiv.org/abs/1512.05858}{arXiv: 1512.05858}.

\bibitem{Dela} J.-P. Delahaye, \emph{Fonctions admettant des cycles d'ordre n'importe quelle puissance de $2$ et aucun autre cycle. (French) [Functions admitting cycles of any power of $2$ and no other cycle].} C. R. Acad. Sci. Paris S\'er. A-B 291 (1980), no. 4, A323--A325

\bibitem{DGS} M.~Denker, C.~Grillenberger and K.~Sigmund,  \emph{Ergodic theory on compact spaces}.
Lecture Notes in Mathematics, Vol. 527. Springer-Verlag, Berlin-New York, 1976.

\bibitem{D05} T.~Downarowicz, \emph{Survey of odometers and Toeplitz flows},
In: Algebraic and Topological Dynamics. Contemporary Mathematics, vol. 385, pp. 7–37.
American Mathematics Society, Providence, 2005.

\bibitem{D04} R.~Dudley, \emph{Real Analysis and Probability},
Cambridge Studies in Advanced Mathematics, 74. Cambridge University Press, Cambridge, 2002.

\bibitem{DOT15} Y. Dong, P. Oprocha, X. Tian,
\emph{On the irregular points for systems with the shadowing property}, preprint,
\href{http://arxiv.org/abs/1510.06956}{arXiv: 1510.06956}.

\bibitem{EKW94} A. Eizenberg, Y. Kifer and B. Weiss,  \emph{Large deviations for $\Z^d$ -actions},
Commun. Math. Phys., \textbf{164} (1994), 433--454.

\bibitem{Gur} B. M. Gurevi\v{c}, Topological entropy of a countable Markov chain. (Russian)
Dokl. Akad. Nauk SSSR, \textbf{187} (1969), 715--718;
English translation in: Soviet Math. Dokl. \textbf{10} (1969), 911--915.


\bibitem{IP84} R. B. Israel, R. R. Phelps, \emph{Some convexity questions arising in statistical
 mechanics},  Math. Scand, \textbf{54} (1984), 133--156.

\bibitem{HKO} G. Harańczyk, D. Kwietniak, P. Oprocha, \emph{Topological structure and entropy of mixing graph maps.} Ergodic Theory Dynam. Systems \textbf{34} (2014), 1587--1614.

\bibitem{Kit} B.P. Kitchens, \emph{Symbolic dynamics. One-sided, two-sided and countable state
Markov shifts}. Universitext. Springer-Verlag, Berlin, 1998.

\bibitem{Kuchta} M. Kuchta, \emph{Shadowing property of continuous maps with zero topological entropy}. Proc. Amer. Math. Soc. \textbf{119} (1993), 641--648.

\bibitem{Kurka} P. Kurka, \emph{Topological and symbolic dynamics}, Cours Sp\'ecialis\'es
[Specialized Courses], vol. 11. Soci\'et\'e Math\'ematique de France, Paris, 2003.

\bibitem{KLO} D. Kwietniak, M. \L{}acka, P. Oprocha, \emph{A panorama of specification-like properties
and their consequences}, Dynamics and Numbers, Contemporary Mathematics, vol. 669, 2016, pp. 155--186.

\bibitem{LO13} J. Li, P. Oprocha,  \emph{Shadowing property, weak mixing and regular recurrence},
J. Dynam. Differential Equations, \textbf{25} (2013), 1233--1249.

\bibitem{LM} D. Lind, B. Marcus, \emph{An introduction to symbolic dynamics and coding},
Cambridge University Press, Cambridge, 1995.

\bibitem{M73} M. Misiurewicz, \emph{Diffeomorphism without any measure with maximal entropy},
Bull. Acad. Pol. Sci. \textbf{21} (1973), 903--910.

\bibitem{M11} T.K.S. Moothathu,
\emph{Implications of pseudo-orbit tracing property for continuous maps on compacta},
Top. Appl. \textbf{158} (2011), 2232--2239.

\bibitem{MO13} T.K.S.~Moothathu, and P.~Oprocha, \emph{Shadowing, entropy and minimal subsystems},
Monatsh. Math. \textbf{172} (2013), 357--378.

\bibitem{Paul} M. E. Paul, \emph{Construction of almost automorphic symbolic minimal flows},
General Topology and Appl., \textbf{6} (1976), 45--56.

\bibitem{PS} C-E. Pfister and W. G. Sullivan,
\emph{Large deviations estimates for dynamical systems without the specification property.
Applications to the $\beta$-shifts}, Nonlinearity, \textbf{18} (2005), 237--261.

\bibitem{RW} D. Richeson  and J. Wiseman, \emph{Chain recurrence rates and topological entropy},
Topology Appl. \textbf{156} (2008), no. 2, 251--261.

\bibitem{Ruelle} D. Ruelle, \emph{Statistical mechanics on a compact set with $\Z$ action
 satisfying expansiveness and specification}, Trans. Amer. Math. Soc., \textbf{185} (1973), 237-251.

\bibitem{Rue} S. Ruette, \emph{On the Vere-Jones classification and existence of maximal measures
for countable topological Markov chains}. Pacific J. Math. \textbf{209} (2003), 366--380.

\bibitem{RuetteSurvey} S. Ruette, \emph{Chaos on the interval - a survey of relationship between the various kinds of chaos for continuous interval maps}, preprint, \href{http://arxiv.org/abs/1504.03001}{arXiv: 1504.03001}.

\bibitem{Sal} I. A. Salama, \emph{On the recurrence of countable topological Markov chains}.
Symbolic dynamics and its applications (New Haven, CT, 1991), 349--360, Contemp. Math.,
135, Amer. Math. Soc., Providence, RI, 1992.

\bibitem{Sig1} K. Sigmund, \emph{Generic properties of invariant measures for Axiom A diffeomorphisms},
 Invent. Math., \textbf{11} (1970), 99--109.

\bibitem{Sig2} K. Sigmund, \emph{On dynamical systems with the specification property},
Trans. Amer. Math. Soc., \textbf{190} (1974), 285--299.

\bibitem{VJ} D. Vere-Jones,  \emph{Geometric ergodicity in denumerable Markov chains},
Quart. J. Math. Oxford Ser., \textbf{13} (1962), 7--28.

\bibitem{Wal} P.~Walters, \emph{On the pseudo-orbit tracing property and its relationship to stability},
The structure of attractors in dynamical systems (Proc. Conf., North Dakota State Univ.,
Fargo, N.D., 1977), pp. 231--244, Lecture Notes in Math., 668, Springer, Berlin, 1978.

\bibitem{Walters}  P.~Walters, \emph{An introduction to ergodic theory}, Springer, Berlin, 2001.

\bibitem{Williams} S.~Williams, \emph{Toeplitz minimal flows which are not uniquely ergodic},
Z. Wahrsch. Verw. Gebiete, \textbf{67} (1984),  95--107.

\bibitem{Yorke} E. M. Coven, I. Kan, J. A. Yorke,  \emph{Pseudo-orbit shadowing in the family of tent maps},
Trans. Amer. Math. Soc. \textbf{308} (1988), 227--241.

\end{thebibliography}
\end{document}